\documentclass[article]{siamart}
\usepackage{caption}
\usepackage{subcaption}
\usepackage{textcomp}
\usepackage{dsfont}
\usepackage[disable]{todonotes} 

\usepackage{lipsum}
\usepackage{amsfonts}
\usepackage{graphicx}
\usepackage{epstopdf}
\usepackage{algorithmic}
\ifpdf
  \DeclareGraphicsExtensions{.eps,.pdf,.png,.jpg}
\else
  \DeclareGraphicsExtensions{.eps}
\fi

\newcommand{\TheTitle}{HIGH ORDER MOMENT MODEL FOR  POLYDISPERSE EVAPORATING SPRAYS TOWARDS INTERFACIAL GEOMETRY DESCRIPTION}
\newcommand{\TheShortTitle}{High order geometrical moment model for polydisperse sprays} 
\newcommand{\TheAuthors}{M. Essadki, S. de Chaisemartin, F. Laurent and  M. Massot}

\headers{\TheShortTitle}{\TheAuthors}

\title{{\TheTitle}}
\author{
  Mohamed Essadki\thanks{IFP Energies nouvelles, 1-4 avenue de Bois-Pr\'eau, 92852 Rueil-Malmaison Cedex  and Laboratoire EM2C, CNRS, CentraleSup\'elec, Universit\'e Paris-Saclay, Grande Voie des Vignes, 92295 Ch\^atenay-Malabry, Cedex - France    (\email{mohamed.essadki@centralesupelec.fr}).}
  \and
  Stephane de Chaisemartin\thanks{IFP Energies nouvelles, 1-4 avenue de Bois-Pr\'eau, 92852 Rueil-Malmaison Cedex - France
    (\email{stephane.de-chaisemartin@ifpen.fr}).}
  \and
  Fr\'ed\'erique Laurent
  \and
  Marc Massot\thanks{Laboratoire EM2C, CNRS, CentraleSup\'elec,Universit\'e Paris-Saclay, Grande Voie des Vignes, 92295 Ch\^atenay-Malabry, Cedex - France and F\'ed\'eration de Math\'ematiques de l'Ecole Centrale Paris, FR CNRS 3487
    (\email{frederique.laurent@centralesupelec.fr - marc.massot@centralesupelec.fr}).}
}

\usepackage{amsopn}

\newcommand{\tens}[1]{{\boldsymbol{#1}}}
\newcommand{\Ug}{\tens{U}_g}

\def\xv{\tens{x}}

\newcommand{\Momsp}[1]{\mathbb{M}_{#1}^{1/2}}

\newcommand{\nq}{\mathit{n}_q}
\newcommand{\negat}{\mathit{n}_q^-}
\newcommand{\U}{\tens{U}}
\newcommand{\cs}{c}
\newcommand{\C}{\tens{\cs}}
\newcommand{\Size}{\phi}
\newcommand{\size}{S}
\newcommand{\radius}{R}

\newcommand{\vecVar}[1]{\tens{#1}}
\newcommand{\nME}{n^{ME}}
\newcommand{\lambEM}{\vecVar{\lambda}^{ME}}

\renewcommand{\theenumi}{\Alph{enumi}}

\newcommand{\Hcurv}{\mathrm{H}}
\newcommand{\Gcurv}{\mathrm{G}}
\newcommand{\Langsize}{\widetilde{\size}}

\newcommand{\deltat}{\Delta t}
\newcommand{\Deltax}{\Delta x}

\newcommand{\lambdaVec}{\tens{\lambda}}

\newcommand{\NEMO}{\textbf{NEMO}}
\newcommand{\f}{f}

\newcommand{\T}{T}

\newcommand{\mom}{\mathrm{m}}
\newcommand{\cmom}{\mathrm{c}}
\newcommand{\Mom}{\tens{\mom}}
\newcommand{\CMom}{\tens{\cmom}}
\newcommand{\ns}{n}
\newcommand{\Alphas}{\alpha}
\newcommand{\Sigmas}{\Sigma}
\newcommand{\Alphad}{\alpha_d}
\newcommand{\Sigmad}{\Sigma_d}
\newcommand{\curv}{\mathrm{k}}

\newcommand{\Pset}{\mathrm{P}}

\newcommand{\Sd}{S}
\newcommand{\ndf}{n}
\newcommand{\Deltat}{\Delta t}
\newcommand{\ds}{\displaystyle}

\newcommand{\pf}{\texttt{p4est}}
\newcommand{\canop}{\texttt{CanoP}}
\newcommand{\Stokes}{\mathrm{St}}

\newcommand{\EvapRate}{\mathrm{R}_S}

\newcommand{\KEvap}{\mathrm{K}}

\ifpdf
\hypersetup{
  pdftitle={\TheTitle},
  pdfauthor={\TheAuthors}
}
\fi

\begin{document}
\maketitle

\begin{abstract}
In this paper we propose a new Eulerian model and related accurate and robust numerical methods, describing polydisperse evaporating sprays, based on high order moment methods in size. The main novelty 
of this model relies on the use of fractional droplet surface moments and their ability to predict some geometrical variables of the droplet-gas interface, by analogy with the liquid-gas interface in interfacial flows. 
Evaporation is evaluated by using a Maximum Entropy (ME) reconstruction. 
The use of fractional moments  introduces some theoretical and numerical difficulties. First, relying on a study of the moment space, we extend the ME reconstruction to the case of fractional moments. 
Then, we propose a new high order and robust algorithm to solve the moment evolution due to evaporation, which preserves the structure of the moment space. 
It involves some negative order fractional moments for which a novel treatment is introduced.
The present model and numerical schemes yield an accurate and stable evaluation
of the moment dynamics with minimal number of variables, as well as computational cost, 
but also provides an additional capacity of coupling with diffuse interface model and transport equation of averaged geometrical interface variables, 
which are essential in order to describe atomization. 
\end{abstract}

\begin{keywords}
  High order moment method, moment space, realizable high order numerical scheme, polydisperse spray, evaporation, entropy maximization, interface geometry.
\end{keywords}

\begin{AMS}
 76T10, 35Q35, 65M08, 65M12, 65M99, 65D99.
\end{AMS}


\section{Introduction}

Modeling ans simulating the fuel injection
in automotive engines and particularly Diesel engine, where the fuel is stored as a liquid phase and injected at high pressure in the 
combustion chamber,  faces major challenges, because of the strongly multi-scale 
character of this two-phase flow problem. It is however essential, since it has a direct impact on the combustion regime and pollutant emissions
and because of the difficulties of getting accurate measurements inside an engine and the high cost 
of experiments \cite{ECN}.

The "Direct Numerical Simulations"  approach 
aims at resolving the mono\-phasic Navier-Stokes equation as well as the interface dynamics and the whole range of scales 
and relies on several classical methods ranging from Volume of Fluid to level set methods and combinations of both,
thus constituting a whole branch of the two-phase literature (see  \cite{vaudor2017,arienti2014} and references therein). However, 
it has a hard time in realistic configurations of high Reynolds and Weber numbers, where it cannot resolve the scale spectrum of both the 
separated-phases two-phase flow close to the injector and the disperse phase,  after atomization leading to a polydisperse 
evaporating spray in the downstream region, where evaporation and combustion are taking place.
Therefore, reduced-order models   
have to be considered \cite{fx2016}, since a full resolution is out of reach for both academic and industrial applications.

Two reduced-order model classes have been designed depending on the flow region. In the 
dense  core region, diffuse interface models (see \cite{drui_PhD,letouze2015,saurel2018}
for a literature review) can be used to simulate interfacial flows with 
lower computational cost, even if the artificial fluid mixing and mesh size lead to some level of diffusion of the interface, and 
we lose important information about the interface geometry. To gain more precision in capturing atomization, 
some models add a transport equation of the interface area density \cite{vallet2001,jay06,lebas2009,devassy2015}, which is not 
not sufficient to describe a key property for applications: the polydispersion in the disperse phase region.

The second reduced-order model class relies on a kinetic approach, which has been widely used to describe accurately a population of particle at mesoscopic level. In this approach, a cloud of droplets is modeled by a number density function (NDF), which 
satisfies a Williams-Boltzmann Equation \cite{williams1958}. 
The internal variables of the NDF provide a statistical description of some relevant physical properties such as the droplet size, velocity and 
temperature. In the following, we focus on size and velocity but temperature can easily be introduced. 
The numerical resolution of the Williams-Boltzmann Equation can be achieved by the Lagrangian Monte-Carlo 
approach. This method is considered to be the most accurate for solving this equation, but 
leads to a high computational cost for unsteady flows and requires complex load-balancing and coupling algorithms in parallel computations. 
An alternative method consists in deriving an Eulerian moment model from the Williams-Boltzmann Equation. In this approach, 
a differential system of a finite set of moments of the NDF is closed by expressing the velocity NDF conditioned in size as a function of the known velocity moments  \cite{kah12,levermore95,levermore96,grad58,sabat2016}. 
For the modeling of the size distribution,  three possible approaches can be used: 1- Multi-fluid models also called sectional methods (see \cite{vieJCP2012,laurent2015} and references therein), where the size phase space is discretized 
into intervals called sections using the conservation of size moments up to two. 
2- Method of moments with interpolative closure \cite{Frenklach2002}, essentially used for soot modeling and simulation, provides a closure of negative as well as fractional moments from the integer ones through a logarithm Lagrangian interpolation. However, such an approach suffers from several issues such as handling multi-modal distributions, preserving the moment space \cite{Mueller2009}, as well as 
 important inaccuracy due to the way the negative order moments are approximated.
3- The third approach involves high order moments on the whole size range, which will be a compact interval in the present study, as well as a procedure in order to reconstruct the NDF. It allows to limit the computational cost and the number of systems of partial differential equations to be solved. We will not consider in the present study quadrature methods such as QMOM or EQMOM \cite{mcgraw1997,marchisio2005,laurent-12,pollack2016} since they lead to either difficulties in evaluating the evaporating flux at zero size or can not reach the integrality of the moment space, but rather focus on Maximum Entropy (ME) reconstruction such as with Eulerian Multi Size Moment (EMSM) and Correlated Size Velocity Moment models in \cite{kah12,massot2010,vieJCP2012,gumprich2016}.
They cover the whole moment space and show  a great capacity in modeling the polydispersion and  evaporation with 
a minimal number of variables. 

So far, the existing models do not provide a unified description of the two regions.
In the present contribution, we propose a new high order moment model for the disperse phase with the capacity of describing the interface geometry statistics \cite{Drew_Geom}. 
Our strategy is to resolve the polydispersion by using a set of variables, which can be identified as averages of the gas-liquid interface geometry. 
We show that some geometrical variables can be expressed as fractional surface-moments of NDF in the disperse zone.
We introduce the mathematics fundamentals of the model and show that we can preserve the advantages of  previous methods \cite{kah12,massot2010} in terms of both accuracy, realizability and robustness, and low computational cost but with a much higher potential in terms of coupling with a diffuse interface model. 
Evaporation flux are evaluated using ME reconstruction based on  fractional moments\footnote{Entropy maximization with fractional moments is used in    \cite{Novi2003,Gzyl2010,Tagliani2011} as a remedy to ill-conditioning  related to a high number of integer moments \cite{Talenti1987}.
The considered set of fractional moments is then recovered from the integer ones, and their orders are optimized to minimize the entropy difference with the real function.
In our contribution, a known and small set of fractional moments is used, deduced from physical considerations, so that the problematic is different.}.
The existence and uniqueness of this convex optimization problem under constraints is given in \cite{mead84} in the case of integer moment. While some elements of proof are to be found in the 
fractional moment case in \cite{Kapur1992},
we propose a generalization of the result in the case of a special set of fractional moments. 
Moreover, we generalize useful properties of the fractional moment space such as canonical moments as well as lower principal representation and quadrature \cite{dette97,dette2002,lasserre2010}. 
These properties are relevant ingredients to design high order realizable schemes and 
algorithms to solve a high order moment system. A new realizable algorithm to solve the evolution of fractional moments due to the evaporation is proposed, which involves negative order of moments and requires an original strategy compared to the integer moment problem. 
The accuracy and robustness of the proposed strategy is then assessed by a careful investigation of the numerical errors as well as a detailed comparison with the original approach in 0D, 1D and 2D academic configurations.
A companion paper  \cite{OGST} aims at implementing the proposed model and numerical strategy into a massively parallel code with adaptive mesh refinement and showing its potential towards realistic engine simulations.
 
The paper is organized as follows. First, the two-phase flow modeling of polydisperse evaporating sprays in a carrier gaseous flow field as well as the original high order moment modeling are introduced; we also recall the classical averaged geometrical description of interfacial flows in order to identify the relevant geometrical averaged variables. Section 3 is devoted to the introduction of the new geometrical high order moments for a polydisperse spray as well as the resulting system of partial differential equations on moments derived from the Williams-Boltzmann Equation. The closure of the system through ME is then presented and its mathematical properties detailed. Once the proposed system is closed, Section 4 and 5 are dedicated to its numerical resolution. While Section 4 is devoted to the transport in phase space, Section 5 focuses on the transport in physical space. Section 6 is eventually concerned with the numerical verification and results in 0D, 1D and 2D, thus assessing the proposed modeling and numerical strategy, before concluding.


\section{Two phase flows modeling} 


\subsection{Kinetic modeling of polydisperse spray and semi-kinetic model}
\label{sec-2}
The spray consists in a cloud of polydisperse droplets, which can be described statistically with the number density function (NDF) $\f^{\Size}(t,\xv,\C,\Size,\T)$. This function represents the probable number of droplets located at position $\xv$, travelling with velocity $\C$ and having temperature $\T$ and  size $\Size$. In general, the size $\Size$ of a spherical droplet can be given by its volume $V$, its surface $\size$ or its radius $\radius$ of the droplet. By considering a spherical form, these three geometrical variables are equivalent $\f^{V}dV=\f^{\size}d\size=\f^{\radius}d\radius$. In the following, we use the surface $\size$ as the size variable. The NDF will be simply noted by $\f$. 
In the following, dimensionless variables are considered, in such a way that $S\in [0,1]$.

For the sake of simplicity and the clarity of the presentation, the derivation of the model is done here in a simplified context: we consider a dilute spray at high Knudsen number and small and spherical droplets at low Weber number, obtained after the atomization. 
Under these assumptions, secondary breakup, coalescence and  collision  can be neglected. 
We also assume that  thermal transfer can be neglected, in such a way that the variable $T$ is no longer considered, that the only force acting on the droplets is the drag force, modeled by the Stokes law and we will mainly focus on one-way coupling. 
Then, the NDF satisfies the following simplified and dimensionless form of the Williams-Boltzmann Equation  \cite{williams1958}:
\begin{equation}
  \partial_{{t}} {\f}
 +\partial_{{\xv}}\cdot\left({\C} {\f}\right)
 +\partial_{{\C}}\cdot\left(\frac{{\Ug}- {\C}}{\Stokes({\size})} {\f}\right)
 +\partial_{{\size}}({\EvapRate} {\f})
=0,
\label{W-B-dimensionless}
\end{equation}
where $\Stokes({\size})=\theta {\size}$ is the Stokes number 
and $\EvapRate$ is the dimensionless evaporation rate, assumed to be constant: $\EvapRate(S)=-\KEvap$.
We refer the readers to the following articles and references therein \cite{laurent2015,kah2015,doisneau14}, showing that such a mesoscopic approach is capable of describing droplet heating, coalescence and break-up and two-way coupling.

The high dimension of the phase space makes its direct discretization too costly for complex applications. Since the resolution of the NDF is not required and only macroscopic quantities of the flows are needed for applications, an Eulerian moment method can be used to reduce the complexity of the problem. 
The closure of the velocity distribution requires a specific treatment, especially when we are concerned with modeling  particle trajectory crossing at high Knudsen numbers and moderate to large inertia. 
For accurate modeling of such phenomenon, one can use high order velocity-moment based on classical equilibrium approaches in kinetic theory, and we rather rely on the Levermore hierarchy \cite{levermore95,levermore96,sabat2016}.
In the present work, we are not concerned with 
these modeling issues, and we will consider a monokinetic velocity distribution \cite{kah2015,sabat2016}:
\begin{equation}
\f(t,\xv,\C,\size)=\ns(t,\xv,\size)\delta(\C-\U(t,\xv,\size)).
\label{eq:monokinetic-dis}
\end{equation}
This closure does not take into account particle trajectory crossing, since only one velocity is defined per position and time. This assumption is valid for low inertia droplets, when the droplet velocities are rapidly relaxed to the local gas velocity \cite{sabat2016}. 
We then derive the following semi-kinetic system from equation \eqref{W-B-dimensionless} by considering moments of order $0$ and $1$ in velocity:
\begin{equation}
\begin{array}{rcl}
\partial_t\ns+\partial_{\xv}\cdot\left(\ns\U\right)&=&\KEvap\,\partial_{\size}\ns,\\
\partial_t\ns\U+\partial_{\xv}\cdot\left(\ns\U\otimes\U\right)&=&
\KEvap\,\partial_{\size}\left(\ns\U\right)+\ns\dfrac{\Ug-\U}{\Stokes(\size)}.
\end{array}
\label{eq:semi-kinetic}
\end{equation}


\subsection{High order size-moment model: closure and moment space}
In the present contribution, we adopt the high order size-moment method with a continuous reconstruction of the size distribution to model a polydisperse spray. This approach, based mainly on \cite{massot2010,kah12,kah2015}, consists in deriving a
dynamical system on a finite set of size-moments of the NDF. The integer size-moments are defined as follows, with $N\ge 1$ for a non dimensional size interval $[0,1]$,
$\Mom=(\mom_0,\ldots,\mom_N)^t$, with $\mom_k=\int_0^{1}\size^k\ns(t,\xv,\size)\mathrm{d}\size$.
For the so-called Eulerian Multi Size Moment (EMSM) model, the droplet velocity is assumed to be independent on the size: $\U(t,\xv)$.
Vi\'e at al. \cite{vieJCP2012} developed a model able to deal with this dependence, but it is beyond the scope of this paper.
The system of the EMSM model obtained from an integration of the semi-kinetic system \eqref{eq:semi-kinetic} over $\size \in [0,1]$ multiplied by $\size^k$ reads:
\begin{equation}
\begin{array}{cccl}
\partial_t\mom_0+\partial_{\xv}.(\mom_0\U)&=&-\KEvap \ns(t,\xv,\size=0)& ,\\
\partial_t\mom_k+\partial_{\xv}.(\mom_k\U)&=&-k\KEvap\ \mom_{k-1}& ,\\[6pt]
\partial_t(\mom_1\U)+\partial_{\xv}.(\mom_1\U\otimes\U)&=&-\KEvap\ \mom_0\U
&+\mom_0\dfrac{\Ug-\U}{\theta}.
\end{array}
\end{equation}
where $k=1\hdots N$.
 The unclosed term $-\KEvap \ns(t,\xv,\size=0)$ expresses the pointwise disappearance flux of droplets through evaporation. Even though, this term is only involved in the first equation, it contributes in the evolution of the other moments in the same way through the terms $-k\KEvap \mom_{k-1}$ 
 \cite{massot2010}. Before proposing a closure of this model, let us recall the definition of the moment space and some useful properties. We denote by $P$ the set of all probability density measures of the interval $[0,1]$. Then the Nth ``normalized" moment space $\mathbb{M}_N$ is defined as follows:
\begin{equation}
\mathbb{M}_N=\left\{\tens{c}_N(\mu), \mu\in P\right\},\hspace{0.2cm}\tens{c}_N=(c_0(\mu),\hdots,c_N(\mu)),\hspace{0.2cm} c_k(\mu)=\int_0^1x^k\mathrm{d}\mu(x).
\end{equation}
We have $c_0=1$, since we use probability density measures. We can also normalize by $\mom_0$ the moment vector $(\mom_0,\hdots,\mom_N)^t$ and obtain $(c_0,\hdots,c_N)^t\in \mathbb{M}_N$, where $c_k=\mom_k/\mom_0$. The Nth ``normalized" moment space is a convex and bounded space.
\begin{definition}
The $N^{\mathrm th}\!$ moment space is defined as the set of vectors $\Mom$,
for which the vector normalized by $\mom_0$ belongs to the $N^{\mathrm th}$ normalized moment space. 
\end{definition}
Considering this definition and some results from \cite{dette97}: if $(\mom_0,\hdots,\mom_N)^t$ is in the interior of the moment space, there exists an infinity of size distributions, which represent this moment vector. In other words, there exists an infinity of size distributions $\ns(\size)$, which are the solution of the following finite Hausdorff moment problem:
\begin{equation}
\mom_k=\int_0^1\size^k\ns(\size)\mathrm{d}\size,\hspace{0.3cm} k=0\hdots N.
\label{moment-sys}
\end{equation} 
Massot et al \cite{massot2010} proposed to use a continuous reconstruction of the 
size distribution through the maximization of Shannon entropy:
\begin{equation}
H(\ns)=-\int_0^1\ns(\size)\ln(\ns(\size))\mathrm{d}\size.
\label{eq:shanonn-entropy}
\end{equation}
The existence and uniqueness of a size distribution $\nME(\size)$, which maximizes the Shannon entropy \cite{Tagliani1999} and is the solution of the finite Hausdorff moment problem \eqref{moment-sys} was proved in \cite{mead84} for the moments of integer order, when the moment vector belongs to the interior of the moment space and the solution is shown to have the form:
\begin{equation}
\ns(\size)=\exp\left(-(\lambda_0+\lambda_1\size+\hdots+\lambda_N\size^N)\right),
\end{equation}
where coefficients $\lambda_k$, are determined from \eqref{moment-sys}. 
The resolution of this nonlinear problem can be achieved by using Newton-Raphson method. 
The limitation of this algorithm when the moment vector is close to the moment space boundary, or equivalently 
when the ME reconstructed size-distribution degenerates to a sum of  Dirac delta functions, is discussed in \cite{massot2010}. 
Vi\'e et al \cite{vieJCP2012} proposed some more advanced solutions to cope with this problem, 
by tabulating the coefficients depending on the moments and by using an adaptive support for the integral calculation, 
which enables an accurate computation of the integral moments when the NDFs are nearly singular.

\subsection{Limitation of such an approach to a disperse phase model}
Even though this high order moment formalism provides some key informations about polydispersion using only one size-section, 
 it is restricted to the disperse phase zone. Coupling such an approach 
 with a separated-phases model requires some complementary informations, which the usual approaches of diffuse interface models can not provide. Indeed, diffuse interface models \cite{drui_PhD,saurel2018} 
consider the interface as a smooth transition layer, where we have lost important informations about the interface geometry. The first step, would be to enrich the diffuse interface models \cite{drui_PhD,essadki2018} 
in order to transport averaged geometrical variables to gain accuracy. 
 Nevertheless, even in the framework of such an enriched diffuse interface model, the coupling of two very different models is usually cumbersome and relies on parameters the described physics will depend on. 
 
 Consequently, we adopt an original strategy and build a high order moment model for the disperse phase, which possesses the same key properties as the EMSM model in terms of accuracy, robustness and computational cost, but 
  involves a different set of variables that are describing the averaged interface geometry, so that we end up with a set of common variables in the two zones and can potentially help in building a single unified model able to capture the proper physics in both zones. In order to introduce 
  the new set of variables, we first have to recall the natural geometrical variables in the separated-phases zone, before extending this description to the disperse phase. 

  \subsection{Geometrical description of interfacial flows}
In many two-phase flow applications, the exact location of each phase is difficult to determine precisely because of the different unpredictable phenomena such as turbulence, interface instabilities and other small scale phenomena. Fortunately, in industrial applications, we are more concerned with the averaged features than to the small details of the flow. Therefore, we can use Diffuse Interface Models 
\cite{drew83,ishii75,drui_PhD,saurel2018} to describe the interface location in terms of probability and averaged quantities based on averaged operators (ensemble averaged, time averaged or volume averaged). In the following, we define some averaged geometrical variables, which can be used to model the interface in separated-phases for a complementary geometrical description. Their definitions are based here on the volume averaged operator following the derivation of Drew \cite{Drew_Geom}. First, we define the phase function $\chi_k(t,\xv)$ for a given phase $k$ by
$\chi_k(t,\xv)=1$ if $\xv\in k$ and $\chi_k(t,\xv)=0$ otherwise.
The volume-averaged operator is defined by: $\overline{(\bullet)}(t,\xv)=\frac{1}{|V|}\int_{V}(\bullet) \chi_k(t,\xv^{\prime})\mathrm{d}\xv^{\prime}$,
where $V\in \mathbb{R}^3$ is a macroscopic space around the position $\xv$, and $|V|$ is the occupied volume, with $\eta^3<<|V|<<L^3$, where $\eta$ (resp. $L$) is the fluctuations (resp. system) length scale.

The Diffuse Interface Models can be obtained by averaging the monophasic fluid equations. The obtained equations involve the volume fraction, which allows to locate the interface up to the averaging scale and is then a first piece of information about the interface geometry:
\begin{equation}
\Alphas_k(t,\xv)=\frac{1}{|V|}\int_{V}\chi_k(t,\xv^{\prime})\mathrm{d}\xv^{\prime}.
\end{equation}
The second variable treated by Drew \cite{Drew_Geom} and used also in other two-phase flow models \cite{vallet2001,lebas2009} is the interfacial area density. The importance of this variable relies mainly on the modeling of exchange terms (evaporation, thermal transfer and drag force) as well as modeling the primary atomization. This variable is interpreted as the ratio of the surface area of an interface contained in a macroscopic volume and this volume.
\begin{equation}
\Sigmas(t,\xv)=\dfrac{1}{|V|}\int_{V}||\tens{\nabla}\chi_k(t,\xv^{\prime})||\mathrm{d}\xv^{\prime}.
\end{equation}
So far, the interface modeling is still incomplete, since no information on the interface shape is being given. In fact, the small details of the interface can not be modeled accurately using only two geometrical variables. Drew proposed to introduce the mean $\Hcurv=\frac{1}{2}(\curv_1+\curv_2)$ and Gauss $\Gcurv=\curv_1\curv_2$ curvatures of the interface in his model, where $\curv_1$ and $\curv_2$ are the two principal curvatures.

These variables are defined only at the interface. Therefore, we need a specific averaging for these interfacial variables. So, we introduce the interfacial averaging operator $\widetilde{(\bullet)}$, defined as follows: $\Sigmas\ \widetilde{(\bullet)}(t,\xv)=\frac{1}{|V|}\int_{V}(\bullet) ||\tens{\nabla}\chi_k(t,\xv^{\prime})||\mathrm{d}\xv^{\prime}.$
The interfacial averaged Gauss and mean curvatures weighted by the interface are defined as follows:
\begin{equation}
    \ds\Sigmas\widetilde{\Hcurv}=\frac{1}{|V|}\int_V\Hcurv||\tens{\nabla}\chi_k(t,\xv^{\prime})||\mathrm{d}\xv^{\prime},\quad
    \ds\Sigmas\widetilde{\Gcurv}=\frac{1}{|V|}\int_V\Gcurv||\tens{\nabla}\chi_k(t,\xv^{\prime})||\mathrm{d}\xv^{\prime}.
\end{equation}
 Drew \cite{Drew_Geom} derived conservative equations for these variables with source terms, which describe the stretch and the wrinkling of the interface. Its derivation is based on a kinematic evolution of an interface, when the interface velocity is given. In real application, the interface velocity can be determined from the diffuse interface model. In the following, we express these geometrical variables in the disperse phase as size-moments of the NDF, and derive a new high order moment model using such variables.
\section{Geometrical high order moment model}


\subsection{Interfacial geometrical variables for the disperse phase}
Let us consider a population of spherical droplets represented by their size distribution $\ns(t,\xv,\size)$. Then, by analogy with the separated phases, we express the averaged geometrical variables: volume fraction, interface area density, Gauss curvature and mean curvature in the disperse phase. The definition of these geometrical variables is based on the phase function $\chi_k$. This function contains all the microscopic information about the interface. In the disperse phase, we use the statistical information about the spherical droplet distribution, which is given by the size distribution $n(t,\xv,\size)$. Considering this function, we define the different geometrical variables in the context of a polydisperse spray as follows:

\begin{enumerate}
\item The \textit{volume fraction} $\Alphas_d$ is the sum of the volume of each droplet 
divided by the contained volume at a given position:
\begin{equation}
\Alphad=\int_0^{1}V(\size)\ns(t,\xv,\size)\mathrm{d}\size, \quad \ds V(\size)=\frac{\size^{3/2}}{6\sqrt{\pi}}.
\end{equation}
\item The \textit{interfacial area density} $\Sigmas_d$ is the sum of the surface of each droplet divided 
by the contained volume at a given position:
\begin{equation}
\Sigmad=\int_0^{1}\size\ns(t,\xv,\size)\mathrm{d}\size.
\end{equation}
\item The two local curvatures are equal for a spherical droplet $\curv_1=\curv_2=\frac{2\sqrt{\pi}}{\sqrt{\size}}$. 
But since we use the mean and Gauss curvatures, we can define two different averaged quantities. 
Let us notice that, in the case of separated phases, the \textit{average mean and Gauss curvatures} 
were defined as an average over a volume and weighed by the interfacial area. In the disperse phase case, this becomes:
\begin{equation}
\ds \Sigmad\widetilde{\Hcurv}_d=\int_0^{1}\Hcurv(\size)\size \ns(t,\xv,\size)\mathrm{d}\size,\quad
\ds \Sigmad\widetilde{\Gcurv}_d=\int_0^{1}\Gcurv(\size)\size\ns(t,\xv,\size)\mathrm{d}\size.
\end{equation}
\end{enumerate}
These four geometrical variables are expressed as fractional moments of the size distribution $\mom_{k/2}=\int_0^1\size^{k/2}\ns(S)\mathrm{d}S$:
\begin{equation}
\begin{array}{rclrrcl}
\Sigmad\widetilde{\Gcurv}_d &=&4\pi \mom_{0/2},\quad&
\Sigmad\widetilde{\Hcurv}_d &=&2\sqrt{\pi} \mom_{1/2},\\
\Sigmad &=& \mom_{2/2},\quad&
\Alphad &=& \frac{1}{6\sqrt{\pi}}\mom_{3/2}.
\end{array}
\end{equation}
These moments can be expressed as integer moments by simple variable substitution $x=\sqrt{\size}$. However, we prefer to hold the droplet surface as the size variable, since we consider a $d^2$ evaporation law, where the evaporation rate $\EvapRate$ is constant.


\subsection{The governing moment equation}
\label{sect:EMSMG}
In this section, we derive from the kinetic equation \eqref{W-B-dimensionless} a high order fractional moment model. This model gives the evolution of the mean geometrical interfacial variables due to transport, evaporation and drag force and reads: 
\begin{equation}
\left\{
 \begin{array}{l@{\ }c@{\ }l@{\ }c@{\ }l@{}l}
  \partial_t\mom_{0/2}    &+& \partial_{\xv}\cdot\left(\mom_{0/2}\U\right)&=&-\KEvap\ns(t,\xv,\size=0) ,\\
  \partial_t\mom_{k/2}&+& \partial_{\xv}\cdot\left(\mom_{k/2}\U\right)&=&-\dfrac{k\KEvap}{2}\mom_{(k-2)/2},\\[4pt]
  \partial_t\left(\mom_{2/2}\U\right)&+& \partial_{\xv}\cdot\left(\mom_{2/2}\U\otimes\U\right)
    &=&-\KEvap \mom_{0/2}\U&+\mom_{0/2}\dfrac{\U_g-\U}{\theta},
 \end{array}\right.
 \label{eq:EMSMG}
\end{equation}
where $k\in\{1,2,3\}$ and $-\KEvap\ns(t,\xv,\size=0)$ represents the pointwise disappearance flux, and the moment of negative order, $\mom_{-1/2}=\int_{0}^{1}\size^{-1/2}\ns(t,\size)\mathrm{d}\size$, naturally appears in the system after integrating by part 
the evaporation term in the Williams-Boltzmann Equation. In the following, the source terms will be closed by a smooth reconstruction of the size distribution through entropy maximisation.

The use of fractional moments introduces a new mathematical framework as well as some numerical difficulties, which require a specific treatment. Some useful mathematical properties of the fractional moments space are discussed in the Appendix \ref{appendix:fractional_moment_space}. 
These results will be used to design realizable numerical schemes, i.e. schemes that ensure the preservation of the moment vector in the moment space. 


\subsection{Maximum Entropy reconstruction}
\label{chap2:ME}
NDF reconstruction through the maximum entropy (ME) provides a smooth reconstruction to close the moment system \eqref{eq:EMSMG} as it was done in the EMSM model. The ME reconstruction consists in maximizing the Shannon entropy  \eqref{eq:shanonn-entropy}, under the condition that the first $N+1$ (here $N=3$) moments of the size distribution are equal to the computed moments:
\begin{equation}\label{eq:MEconstraints}
\mom_{k/2}=\int_0^1S^{k/2}n(S)\mathrm{d}S, \quad k=0\hdots N.
\end{equation}
In the Appendix \ref{SM:ME-existence and uniqueness}, the existence and uniqueness of the ME distribution function is shown, as presented in the following Theorem.
The proof is based on the method of Lagrange multipliers, using some ideas of Mead and Papanicolaou \cite{mead84}, who already showed this result in the case of integer moments, but in a different and simplified way.
Indeed, in our case, we only use the definition of the fractional moment space, given in the Appendix \ref{appendix:fractional_moment_space}.
\begin{theorem}
If the vector $\Mom_{N}=(\mom_0,\mom_{1/2},..,\mom_{N/2})$ belongs to the interior of the Nth fractional moment space, then there exists a unique NDF $\nME$ maximizing the Shannon entropy defined by \eqref{eq:shanonn-entropy} under the constraints \eqref{eq:MEconstraints}.
Moreover, it has the following form:
\begin{equation}
\nME(\size)=\exp\left(-\lambda_0-\sum\limits_{i=1}^{N}\lambda_i\size^{i/2}\right).
\end{equation}
\end{theorem}
Moreover, in the Appendix \ref{SM:ME-algo}, an algorithm is given to compute the parameters $\lambda_i$ from the fractional moments.



\section{Evaporation source term}
In this section, we focus on the numerical resolution of the evaporation source terms. 
The kinetic equation then reads:
\begin{equation}
\partial_tn-\partial_S(Kn)=0.
\label{eq:evap-kinet}
\end{equation}
The integration of this equation multiplied by the vector $(1,\size^{1/2},\size,\size^{3/2})^t$ yields to a system of ordinary 
differential equations for the moments defined on $[0,1]$:
\begin{equation}
\left\{
 \begin{array}{rclrrcl}
d_t\mom_{0/2}&=&-\nME|_{S=0},\quad &
d_t\mom_{1/2}&=&-\frac{\KEvap}{2}\mom_{-1/2},\\[7pt]
d_t\mom_{2/2}&=&-\KEvap \mom_{0/2},\quad &
d_t\mom_{3/2}&=&-\frac{3\KEvap}{2}\mom_{1/2},\\
 \end{array}\right.
 \label{eq:Moment_equation-evap}
\end{equation}
where the $\nME|_{S=0}(\mom_{0/2},\mom_{1/2},\mom_{2/2},\mom_{3/2})$ is obtained by maximum entropy reconstruction.
Solving this system using classical integrators such as Euler or Runge-Kutta methods does not ensure the preservation of the moments in the moment space \cite{massot2010}. This property is essential for robustness and accuracy to reconstruct a positive NDF. 

\subsection{Exact kinetic solution through the method of characteristics}
The exact solution of the NDF evolution, in the case of $d^2$ evaporation law, can be easily obtained  by 
solving analytically the kinetic equation \eqref{eq:evap-kinet}:
\begin{equation}
n(t,S)=n(0,S+\KEvap t).
\label{eq:exact_kinet_solution_d2}
\end{equation}
For more general evaporation law, when the evaporation rate $\EvapRate(S)$ is a smooth function of the size, the kinetic equation can be solved by using the method of characteristics. Indeed, by multiplying the kinetic equation by $\EvapRate(S)$, we obtain the following equation in the function $\Gamma(t,S)=\EvapRate(S)n(t,S)$:
\begin{equation}
\partial_t\Gamma(t,S)+\EvapRate(S)\partial_S(\Gamma(t,S))=0.
\end{equation}
Then, for a given initial time $t_0$ and size $\Sd_o$, we define the one variable function $g(t)=\Gamma(t,\Langsize(t;t_0,\Sd_o))$, such that $\Langsize(t;t_0,\Sd_o)$ is the characteristic curve verifying:
\begin{equation}
\left\{
\begin{array}{rcl}
\dfrac{d\Langsize(t;t_0,\Sd_o)}{dt}&=&\EvapRate(\Langsize(t;t_0,\Sd_o)),\\
\Langsize(t_0;t_0,\Sd_o)&=&\Sd_o.
\end{array}\right.
\label{eq:Lagrang-size-evol}
\end{equation}
It is a simple fact that the derivative of $g(t)$ vanishes, which implies $g(t)=g(t_0)$ and then
$ 
\Gamma(t,\Langsize(t;t_0,\Sd_o))= \Gamma(t_0,\Sd_o) $.
Finally, we obtain the exact solution of the size distribution as follows:
\begin{equation}
\ndf(t,\Sd)=\dfrac{\EvapRate(\Langsize(t_0;t,\Sd))}{\EvapRate(\Sd)}\ndf(t_0,\Langsize(t_0;t,\Sd)).
\label{eq:exact-kinetic-sol2}
\end{equation}


\subsection{Inefficiency of the original evaporation scheme}
\label{sect:EMSM_alg}
It is possible to use the exact solution \eqref{eq:exact_kinet_solution_d2} or \eqref{eq:exact-kinetic-sol2} of \eqref{eq:evap-kinet}, at the kinetic level, to compute directly the moments. In this case, 
a Gauss-Legendre quadrature can approximate the moment-integrals. In \cite{mead84}, the authors showed that $24$ quadrature points are needed to have an accurate estimation of such integrals. However, when the evaporation rate $\EvapRate$ 
is not a constant, we will also need to solve the ODE \eqref{eq:Lagrang-size-evol} for each quadrature points. For this reason, we consider this 
direct method in computing the moments will be very costly in terms of CPU time for realistic simulations. 

Massot et al \cite{massot2010} proposed a realizable algorithm to solve the evaporation moment system in 
the case of integer moments. This algorithm is based on evaluating the disappearance evaporation flux, that is  the droplets that will completely evaporate during one time step, then use a simplified quadrature to 
evaluate the kinetic evolution of the moments. This algorithm reads in the case of fractional moments:
\begin{itemize}
\item From the moment vector $\Mom(t_n)=(\mom_{0/2},\mom_{1/2},\mom_{2/2},\mom_{3/2})^t$, we reconstruct 
the NDF by using the ME algorithm. Then, we compute the disappearance fluxes of the droplets which will be totally evaporated at the next step:
\begin{equation}
\tens{\Phi}_-(t_n)=\int_0^{\KEvap\Deltat}n_{ME}(t_n,s)\begin{pmatrix}
1\\
s^{1/2}\\
s\\
s^{3/2}
\end{pmatrix} \mathrm{d}s.
\label{eq:disappearance_flux}
\end{equation}
\item Using the Product-Difference algorithm \cite{gordon1968}, we compute the abscissas $S_j\geq \KEvap \Deltat$ and the weights $w_j\geq 0$ of the quadrature corresponding to of the moments\footnote{We can also use hybrid approach, in this case we take into account the evaporation fluxes coming from the right section: $\Mom(t_n)-\tens{\Phi}^{(i)}_-(t_n)+\tens{\Phi}^{(i+1)}_-(t_n)$, where $i$ indexes the size-section.} $\Mom^{[\KEvap \Deltat,1]}(t_n)=\Mom(t_n)-\tens{\Phi}_-(t_n)$. This corresponds to the resolution:
\begin{equation}
\mom^{[\KEvap \Deltat,1]}_{k/2}(t_n)=\sum\limits_{j=1}^{2}w_jS_j^{k/2}, \quad k\in\{0,1,2,3\},
\label{eq:al_emsm_quad}
\end{equation}
where $\Mom^{[\KEvap \Deltat,1]}(t_n)$ is the moment vector on $[\KEvap\Deltat,1]$.
\item Finally, we calculate the moments at the next step:
\begin{equation}
\mom_{k/2}(t_n+\Deltat)=\sum_{j=1}^{2}w_j(S_j-\KEvap \Deltat)^{k/2}.
\label{eq:al_size_evol}
\end{equation} 
\end{itemize}
However, in the case of the fractional moments, the present algorithm can not predict the correct kinetic evolution 
of the NDF. We illustrate this problem in Figure \ref{fig:emsm-algo}, where we 
compare the maximum entropy reconstructed NDF of the moment computed with the above algorithm and  
the exact kinetic solution of the NDF. It is clear that the solution obtained by the algorithm diverges 
from the exact solution. In the following, we will provide the reader with an explanation of the inefficiency of this algorithm in the 
case of fractional moments and  propose a new specific solution.
\begin{figure}
\centering
\includegraphics[width=.5\linewidth]{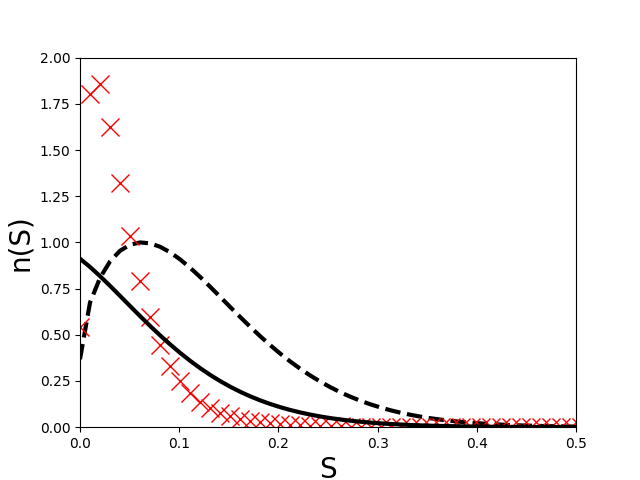}
\caption{Initial size distribution (dashed line) and the reconstructed size distribution at $t=0.1$: kinetic solution (solid line), original evaporation scheme for fractional moments (cross).}
\vspace{-0.5cm}
\label{fig:emsm-algo}
\end{figure}
\subsection{Involvement of negative moment order}
\label{subsect:adapted_evap_algor}
In order to understand the inefficiency of the previous algorithm proposed in section \ref{sect:EMSM_alg}, we propose 
to consider again a constant evaporation rate $\EvapRate(\Sd)=-\KEvap$. The moment solution expressed directly from 
the exact kinetic solution reads:
\begin{equation}
\begin{array}{r@{}c@{}l}
\mom_{k/2}(t\!+\! \Deltat)&=&\int_0^1S^{k/2}\ndf(t\!+\!\Deltat,S)\mathrm{d}S 
		=\int_0^1S^{k/2}\ndf(t,\Sd\!+\!\KEvap\Deltat)\mathrm{d}\Sd\\[6pt]
	&=&\int_{\KEvap\Deltat}^{1+\KEvap\Deltat}(S\!-\!K \Deltat)^{k/2}\ndf( t,\Sd)\mathrm{d}\Sd
		=\int_{\KEvap\Deltat}^{1}(S\!-\!K \Deltat)^{k/2}\ndf( t,\Sd)\mathrm{d}\Sd.	 
\end{array}
\label{eq:Kinetic-moment-sol}
\end{equation}
In the previous equality, we consider that $\ndf( t,\Sd)=0$ for $\Sd\geq 1$. In the following, we show that the 
updated moments can be written as a function of an infinite set of moments on the support 
$[\KEvap \Deltat,1]$. 
\begin{lemma}
For all positive integer $k$ and for all $x \in [-1,1]$ the function $f:x\rightarrow (1-x)^{k/2}$ admits a power series which converges normally to the function $f$:
\begin{equation}\label{eq:serie}
(1-x)^{k/2}=\sum_{n=0}^{+\infty}a_n^{k/2}x^n,
\end{equation}
\end{lemma}
\begin{proof}
The  even integer case $k=2m$ is trivial with $a_n^{k/2}=(-1)^n\binom{m}{n}\mathds{1}_{[0,m]}(n)$. 
Let us consider the case where $k=2m+1$ is an odd number. The function $f$ admits a power series of the form \eqref{eq:serie}.
Moreover, its coefficients can be written: $a_0^{k/2}=1$ and
\begin{equation}
a_n^{k/2}=(-1)^n \frac{(m+1/2)(m-1/2)\dots(m+1/2-n+1)}{n!},
\quad n>0.
\end{equation}
By using the Stirling's approximation, for $n\ge m+3/2$, we can show the following equivalence relation when $n$ tends to infinity.
\begin{equation}
|a_n^{k/2}| = \left|(-1)^{m+1}\frac{(2(n-m-1))!(2m+1)!}{(n-m-1)!m!2^{2n-1}n!}\right|
\sim \frac{(2m+1)!}{4^m\sqrt{\pi}m!}\frac{1}{n^{3/2+m}}.
\end{equation}
Therefore the series $\sum |a_n^{k/2}|$ is convergent for any integer $k\geq0$, thus concluding the proof. 
\end{proof}

We deduce that $\sum\limits_{n=0}^{\infty}a_n^{k/2}(\KEvap \Deltat)^ns^{k/2-n}$ converges normally to $(s-\KEvap \Deltat)^{k/2}$ for \\$s\geq \KEvap \Deltat$. Thus, we can invert the sum and the integral in the moment expression:
\begin{equation}
\begin{array}{rcl}
\mom_{k/2}(t+\Deltat)&=&\int_{\KEvap \Deltat}^1(s-\KEvap \Deltat)^{k/2}n(t,s)\mathrm{d}s,\\[9pt]
 &=&\sum\limits_{n=0}^{\infty}a_n^{k/2}(\KEvap \Deltat)^n\int_{\KEvap \Deltat}^1s^{k/2-n}\ndf(t,s)\mathrm{d}s,\\[9pt]
 &=& \sum\limits_{n=0}^{\infty}a_n^{k/2}(\KEvap \Deltat)^n \mom^{[\KEvap \Deltat,1]}_{k/2-n}(t).
\end{array}
\label{eq:series_algo}
\end{equation}
where $\mom^{[\KEvap \Deltat,1]}_{l/2}=\int_{\KEvap \Deltat}^1s^{k/2-n}\ndf(t,s)\mathrm{d}s$ is a fractional moment of support 
$[\KEvap \Deltat,1]$.

Equation \eqref{eq:series_algo} shows that the fractional moments at $t+\Deltat$ depend on the four fractional 
moments  ($m^{[\KEvap \Deltat,1]}_{l}$ where $l=0/2,1/2\hdots3/2$) and on an infinite set of the moments of support $[\KEvap \Deltat,1]$ ($m^{[\KEvap \Deltat,1]}_{l}$ where $l=-1/2,-2/2,\hdots-\infty$). In the case of the EMSM model, where only integer moments are used, the same expansion of the exact kinetic solution of the integer moments involves only the moments of the support $[\KEvap \Deltat,1]$ and of order $l\in\{0,1,2,3\}$. For this reason, the evolution of the integer moments can be evaluated exactly by translating the abscissas \eqref{eq:al_size_evol} of the quadrature \eqref{eq:al_emsm_quad}. The same strategy for fractional moments leads to the divergence of the method observed previously.  Since we aim at describing the evolution of the moments through the evolution of a minimal number of quadrature nodes, we propose 
a new quadrature,  suitable for fractional moments, such that the updated moments in \eqref{eq:al_size_evol} can approximate accurately the exact kinetic evolution \eqref{eq:series_algo}.

\subsection{A specific quadrature for negative moment order}

Our  objective is to find an adequate quadrature with the lowest possible quadrature number $\nq$, such that the following approximation:
\begin{equation}
\mom_{k/2}(t+\Deltat)\approx \sum\limits_{j=1}^{\nq}w_j(S_j-K\Deltat)^{k/2},
\label{eq:approx_algo_nemo}
\end{equation}
is accurate. More precisely, we would like to find a quadrature such that the difference, for $k\in\{0,1,2,3\}$ 
\begin{equation}
\begin{array}{r@{\ }c@{\ }l}
\epsilon_{k/2}(\Deltat)&=&\mom_{k/2}(t+\Deltat)-\sum\limits_{j=1}^{\nq}w_j(S_j-\KEvap \Deltat)^{k/2}\\
 &=&\sum\limits_{n=0}^{\infty}a_n^{k/2}(\KEvap \Deltat)^n \left( m^{[\KEvap \Deltat,1]}_{k/2-n}(t)-\sum\limits_{j=1}^{\nq}w_jS_j^{k/2-n}\right),
\end{array}
\label{eq:algo-error}
\end{equation}
is at least $o(\Deltat)$ to ensure the convergence of the numerical scheme. 
It is a difficult task to prove such a result, but we provide here a first result is this direction.
Our strategy consists in using a quadrature, which cancels a finite set of the first terms in the sum \eqref{eq:algo-error}, by enforcing $E_{p/2}=\mom^{[\KEvap\Deltat,1]}_{p/2}(t)-\sum\limits_{j=1}^{\nq}w_j\Sd_j^{p/2}$ to be zero for $p\in\{-2\negat,\dots,2,3\}$.
Thus $2\nq$ fractional moments on $[\KEvap\Deltat,1]$ are considered, with $\nq=\negat+2$, from order $-\negat$ to $3/2$.
As soon as $\negat\ge 0$, this is sufficient to have $\epsilon_{0/2}=0$ and $\epsilon_{2/2}=0$.
The following Lemma shows the existence and uniqueness of such quadrature and gives the corresponding value of the error terms $\epsilon_{k/2}(\Deltat)$.
\begin{lemma}
Let $\left(\mom^{[\KEvap\Deltat,1]}_{p/2}\right)_{p\in\{-2\negat,\dots,3\}}$ be a set of fractional moments on $[\KEvap\Deltat,1]$ in the interior of the moment space and $\negat$ a non negative integer. 
There exists a unique set of abscissas $(S_j)_{j\in\{1,\dots,\nq\}}$ and weights $(w_j)_{j\in\{1,\dots,\nq\}}$, with $\nq = \negat+2$ such that 
\begin{equation}\label{eq:quad}
\forall k\in\{-2\negat,\hdots,3\} \qquad \mom^{[\KEvap \Deltat]}_{k/2}=\sum\limits_{j=1}^{\nq}w_jS_j^{k/2}.
\end{equation}
Moreover one have $w_j=w^{\prime}_jr_j^{2\negat}$ and $S_j=r_j^2$, where $r_j$ and $w^{\prime}_j$ are the abscissas and weights of the Gauss quadrature corresponding to the moments $m_k=\mom^{[\KEvap\Deltat,1]}_{k/2+\negat}$ for $k\in\{0,\dots,2\nq-1\}$.
\\
Using this quadrature, the error terms defined by \eqref{eq:algo-error} can be written:
\begin{equation}
\begin{array}{rclrrcl}
\epsilon_{0/2}(\Deltat)&=&0,\quad &
\epsilon_{1/2}(\Deltat)&=&\left(\sum\limits_{n=\nq-1}^{+\infty}a_n^{1/2}\bar E_{(2n-1)/2}\right)(\KEvap\Deltat)^{1/2},\\[9pt]
\epsilon_{2/2}(\Deltat)&=&0,\quad &
\epsilon_{3/2}(\Deltat)&=&\left(\sum\limits_{n=\nq}^{+\infty}a_n^{3/2}\bar E_{(2n-3)/2}\right)(\KEvap\Deltat)^{3/2},
\end{array}
\label{eq:errors_negative_quad}
\end{equation}
where $\bar E_{p/2}= ({\KEvap\Deltat})^{p/2}\left( \mom^{[\KEvap\Deltat,1]}_{-p/2}(t)-\sum\limits_{j=1}^{\nq}w_j\Sd_j^{-p/2}\right)$ is bounded by $\mom^{[\KEvap\Deltat,1]}_{0/2}$. 
\end{lemma}

\begin{proof}
To be able to compute the weights $w_j$ and abscissas $S_j$ of the quadrature \eqref{eq:quad}, the fractional moments are transformed to integer ones, with non negative orders.
Indeed, 
using the change of variable $r=\sqrt{S}$, $E_{p/2+\negat}$ can be written, for $k\in\{0,\dots,2\nq-1\}$:
\begin{equation}\label{eq:transfquad}
E_{p/2+\negat} \!=\! 
\int_{\KEvap \Deltat}^1\!s^{k/2-\negat}n(s)\mathrm{d}s \!-\! \sum\limits_{j=1}^{\nq}w_jS_j^{k/2-\negat}
\!=\! \int_{\sqrt{\KEvap \Deltat}}^1\!r^k{\mu}(r)\mathrm{d}r \!-\! \sum\limits_{j=1}^{\nq}w^{\prime}_jr_j^{k},
\end{equation}
with ${\mu}(r)=(2rn(r^2))/(r^{2\negat})$ for $r\in [\sqrt{\KEvap \Deltat},1]$, $r_j = \sqrt{S_j}\in[\sqrt{\KEvap \Deltat},1]$ and $w^{\prime}_j=w_jr_j^{-2\negat}\ge 0$.
Then $w^{\prime}_j$ and $r_j$ are defined as the unique weights and abscissas of the Gauss quadrature of order $\nq$ corresponding to moments $m_k=\mom^{[\KEvap\Deltat,1]}_{k/2+\negat}$ of $\mu$ on $[\sqrt{\KEvap \Deltat},1]$ for $k\in\{0,\dots,2\nq-1\}$.

By using this quadrature in equation \eqref{eq:algo-error}, we cancel the $k+\negat$ first terms in the series, thus obtaining \eqref{eq:errors_negative_quad}.
The quantities $\bar E_{p/2}$ are bounded, since $S_j\ge \KEvap\Deltat$ and
\begin{equation*}
|\bar E_{p/2}|\leq \max\left\{({\KEvap\Deltat})^{p/2} \mom^{[\KEvap\Deltat,1]}_{-p/2}, 
({\KEvap\Deltat})^{p/2}\sum\limits_{j=1}^{\nq}w_j\Sd_j^{-p/2}\right\} 
\leq  \mom^{[\KEvap\Deltat,1]}_{0/2}.
\end{equation*}
\ 
\end{proof}

However, it is still a rough bound, since this only leads to 
$\epsilon_{1/2}=O(\Deltat^{1/2})$, when $\Deltat$ goes to zero, which is not sufficient 
to prove the convergence. 
An accurate estimation of the Gauss quadrature errors is needed to improve the bound, which should be investigated further in future works.

It seems now clear that the time evolution of the fractional moments depends on other moments than the ones involved in the moment system \eqref{eq:EMSMG} and we cannot restrict the quadrature to the mere moments of positive order, which also corresponds to taking $\negat=0$, if we want to be able to provide a good approximation of the moment evolution. This is thus the key issue with fractional moment evolution, since 
we need at least to consider some moments of negative order in the quadrature, i.e.  $\negat>0$. 
After a series of test-cases, we would be tempted to conjecture that for $\negat=1$, where two supplementary moments of negative order ($m^{[\KEvap\Deltat,1]}_{-1/2}$ and  $m^{[\KEvap\Deltat,1]}_{-2/2}$) are also represented by the quadrature, the solution approximates accurately the exact kinetic solution.  We were not able to complete the proof of such a conjecture and some of these results are presented in Section \ref{sect:Results}. 
It seems that the terms $\bar E_{l/2}$ where $l=1,2$ in equation \eqref{eq:errors_negative_quad} cause the divergence of the standard algorithm proposed in section \ref{sect:EMSM_alg}. 
This is an important result, since we need only a total of three quadrature nodes $\nq=3$ to cancel the terms $E_{l/2}$ for $l=1,2$ and to capture correctly the kinetic evolution. Besides, we can use more quadrature points to increase the precision by choosing $\negat\geq2$. 
Let us underline that the proper approximation and closure of the negative moments is here a key issue. 

\paragraph*{\textbf{New algorithm}:}
according to the last results, we propose a $4$-steps algorithm. This algorithm
is named \NEMO\ (Negative Moments) algorithm and described below:
\label{algo:new_evap_algorithm}
\begin{enumerate}
\item We reconstruct $\nME(S)$ corresponding to the moment vector $\Mom(t_n)$ by the ME algorithm, 
then we calculate the disappearance flux as in \eqref{eq:disappearance_flux},
\item We calculate the negative order moments in the interval $[\KEvap\Deltat,1]$
\begin{equation}
\mom_{-a/2}^{[\KEvap\Deltat,1]}=\int_{\KEvap \Deltat}^{1}s^{-a/2}\nME(s)\mathrm{d}s,
\label{eq:negative_moment}
\end{equation} 
for $a=1,\hdots,2\negat$, where $2\negat\geq 2$ is the number of additional moments of negative order used in this algorithm and chosen by the user. The other moments of positive order are computed using the disappearance flux: 
\begin{equation}
\mom_{k/2}^{[\KEvap \Deltat,1]}=\mom_{k/2}(0)-\Phi_{-,k/2}(t_n),\quad k=0,\hdots,3,
\end{equation} 
where $\tens{\Phi}$ is the disappearance flux given in \eqref{eq:disappearance_flux}.
\item The abscissas $S_j=r_j^2\in [\KEvap \Deltat,1]$ and the weights $w_j=w^{\prime}_jr_j^{2\negat}$ of the quadrature corresponding to the moments $m^{[\KEvap \Deltat,1]}_{p/2}$ for $p=-2\negat,\hdots,3$ are computed using the Product-Difference Algorithm \cite{gordon1968} corresponding to the moments $m_k=\mom^{[\KEvap\Deltat,1]}_{k/2+\negat}$ for $k\in\{0,\dots,2\nq-1\}$, leading to
\begin{equation}
m^{[\KEvap \Deltat,1]}_{p/2}=\sum\limits_{j=1}^{\nq}w_jS_j^{p/2}, \qquad p=-2\negat,\hdots,3
\label{eq:LPR_frac}
\end{equation}
 \item Finally, we calculate the updated moments as follows:
\begin{equation}
\mom_{k/2}(t_n+\Deltat)=\sum_{j=1}^{\nq}w_j(S_j-\KEvap\Deltat)^{k/2}.
\label{algo:NEMO-step4}
\end{equation}
\end{enumerate}
The singularities of the negative moment-integrals, when $\Deltat$ is very small, limits the use of high values of $\negat$. But in practice $\KEvap\Deltat>1.e-4$ and we will show that the choice of $\negat=1$ or $\negat=2$  are sufficient to obtain an accurate solution. In these cases, the integral computation of the negative order moments can be achieved correctly with $24$ Gauss-Legendre quadrature points. For more complex evaporation laws, the algorithm can be 
straightforwardly generalized by computing the Lagrangian evolution of the abscissas. In other words, equation \eqref{algo:NEMO-step4} becomes:
\begin{equation}
\mom_{k/2}(t_n+\Deltat)=\sum_{j=1}^{\nq}w_j\Langsize(t_n+\Deltat;t_n,S_j)^{k/2},
\end{equation}
and the integral of negative moments in \eqref{eq:negative_moment} becomes:
\begin{equation}
\mom_{-a/2}^{[\KEvap\Deltat,1]}=\int_{\Langsize(t_n;t_n+\Deltat,0)}^{1}s^{-a/2}\nME(s)\mathrm{d}s.
\end{equation}
\section{Numerical resolution of the moment governing equations}
\label{sect:Evap_drag}
In this section, we present the global numerical resolution of system of equations \eqref{eq:EMSMG}. 
To resolve this system, we adopt operator splitting techniques \cite{doisneau14,descombes14} to separate the resolution of the convective transport part and the source terms.  For the transport part, we extend the kinetic finite volume scheme \cite{kah12} used for the resolution of the EMSM model to the present case of fractional moments, while ensuring robust and realizable resolution. The obtained kinetic finite volume scheme is detailed in \ref{app:Transport-scheme}. The source terms consists of two parts: evaporation and drag force. In this paragraph, we present a coupled solver for the spray evolution 
under evaporation and drag force. The corresponding system of equations reads:
 \begin{equation}
 \left\{
 \begin{array}{rcl}
d_t\mom_{0/2}&=&-\KEvap \ndf(S=0),\\
d_t\mom_{1/2}&=&-\dfrac{\KEvap}{2}\mom_{-1/2},\\
d_t\mom_{2/2}&=&-\KEvap \mom_{0/2},\\
d_t\mom_{3/2}&=&-\dfrac{3\KEvap}{2}\mom_{1/2},\\
d_t\left(\mom_{2/2}\U\right)&=&-\KEvap \mom_{0/2}\U+\mom_{0/2}\dfrac{\Ug-\U}{\theta}.
 \end{array}\right.
 \label{eq:drag-evap}
\end{equation}
Since the first four equations do not depend on the last one, the updated moments can be computed using \NEMO\ algorithm. For the last equation, we use the method developed in \cite{vieJCP2012} to solve the velocity evolution due to the drag force and evaporation.

The momentum evolution is conducted in two steps: first, we remove the part of the droplets, which will completely evaporate during the time interval $[t_n,t_{n+1}]$, by evaluating the disappearance fluxes of the moments and  momentum:
\begin{equation}
\mathcal{U}^{[\KEvap\Deltat,1]}=\mathcal{U}-\begin{pmatrix}
\tens{\Phi}_-\\
\Phi_{-,2/2}\U
\end{pmatrix},
\end{equation}
where $\mathcal{U}=(\mom_{0/2},\hdots,\mom_{3/2},\mom_{2/2}\U^t)^t$ and $\tens{\Phi}_-$  is the disappearance flux vector of the moments \eqref{eq:disappearance_flux}.
Then, the quadrature abscissas $S_i(t_n)$ and weights $w_i(t_n)$ are computed from \eqref{eq:LPR_frac} from the moments on the support $[\KEvap\Deltat,1]$. 
The computation of the moments is achieved with \eqref{algo:NEMO-step4}. 
To update the momentum in this last step, we consider a size-velocity correlation: at time $t_n$, we attribute the velocity $\C_i(t_n)=\U(t_n)$ to each abscissa $S_i(t_n)$. 
Then, the evolution on the time step $[t_n,t_{n+1}]$ of $\C_i$ and $S_i$ are computed by solving the following system:
\begin{equation}
\dfrac{d\C_i}{dt}=-\dfrac{\Ug-\C_i}{\theta S_i}, \qquad
\dfrac{dS_i}{dt}=-\KEvap.
\label{lagrang-evap}
\end{equation}
The final momentum is computed at $t=t_{n+1}$ as follows:
 \begin{equation}
(\mom_{2/2}\U)(t_{n+1})=\sum\limits_{i=1}^{\nq}w_iS_i(t_{n+1})\C_i(t_{n+1}).
\end{equation}
The method can be generalized to more complex evaporation law by replacing $-\KEvap$ in equation \eqref{lagrang-evap} by a general evaporation rate $\EvapRate(S)$.


\section{Numerical results}
\label{sect:Results}
This section is dedicated to some representative test-cases and analysis of numerical results, to verify the robustness and the accuracy of the proposed numerical schemes. In the first part, we test the new evaporation algorithm in the case of $d^2$ law with $\EvapRate(\Sd)=-\KEvap=-1$ for two different initial conditions. A complementary evaporation test in the case of $\EvapRate(\Sd)=-(1+0.5\Sd)$ is also presented to test the algorithm accuracy with a non constant evaporation rate. The second part focuses on two cases of transport in $1D$. First, we test the accuracy of the kinetic finite volume schemes dedicated to the transport. Second, a critical case of a $\delta$-shock is performed to evaluate the robustness of these schemes. Finally, a 2D case of an evaporating spray in the presence of a steady gaseous flow field, given by Taylor-Green vortices, is presented, in order to qualify the robustness and accuracy of the method compared to EMSM with its original numerical schemes \cite{massot2010,kah12}.


\subsection{Evaporation in 0D simulation}
\subsubsection{Evaporation with $d^2$ law for an initial smooth NDF}
We consider an initial NDF in the form of the ME-reconstruction NDF, which is the same initial condition as the one used in Figure \ref{fig:emsm-algo}: 
\begin{equation}
n^0(S)=\exp(-16(S^{1/2}-1/4)^2(S^{1/2}+1)).
\end{equation}
\begin{figure}
\begin{subfigure}{.48\textwidth}
  \centering
  \includegraphics[width=.95\linewidth]{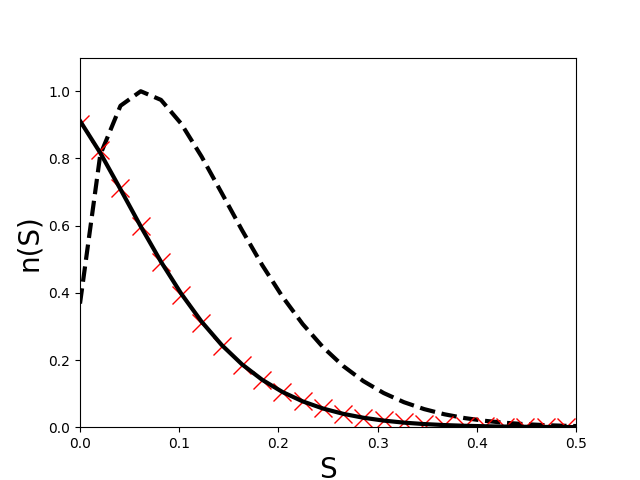}
  \label{fig:ndf_1-t02}
\end{subfigure}%
\begin{subfigure}{.48\textwidth}
  \includegraphics[width=.95\linewidth]{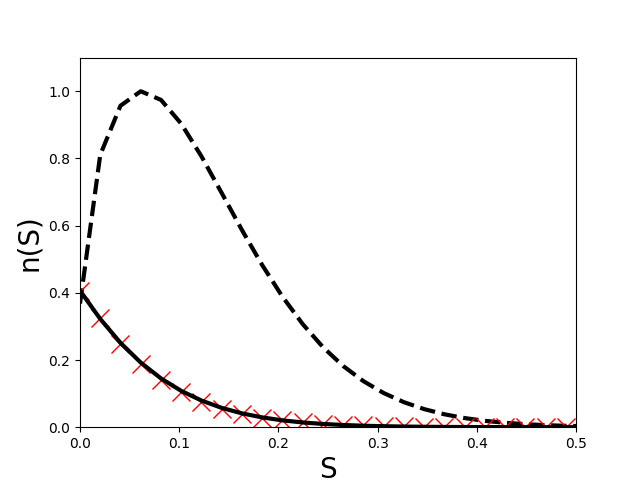}
  \label{fig:ndf_1_t04}
\end{subfigure}
 \caption{ME-reconstructed NDF solution of \NEMO\ with $\negat=1$ (cross), exact kinetic NDF (solid line) and the initial distribution (dashed line), at time $t=0.1$ (left) and $t=0.2$ (right).}
\label{fig:ndf-1}
\vspace{-0.4cm}
\end{figure}

\begin{figure}[!htb]
    \vspace{-0.5cm}
    \begin{minipage}{.48\textwidth}        
  \includegraphics[width=.95\linewidth]{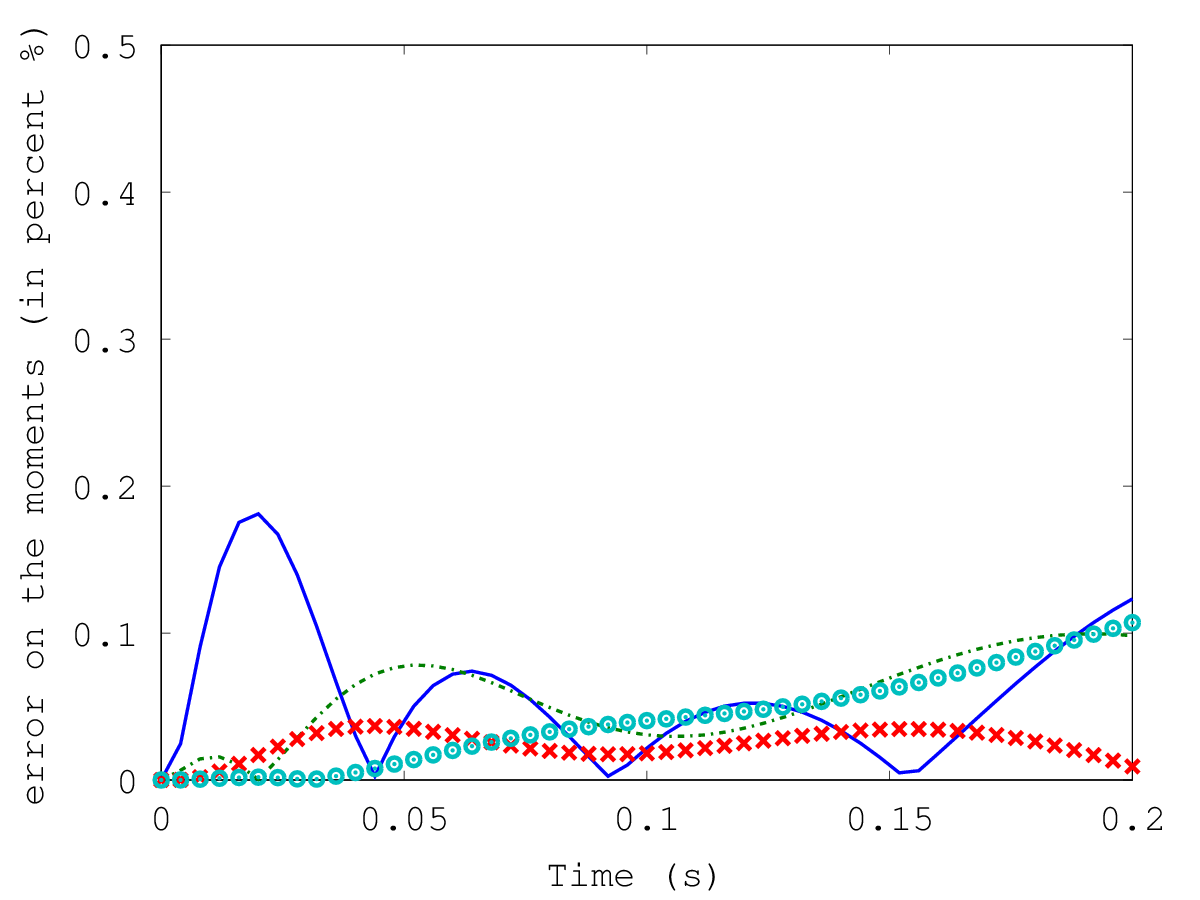}
 \caption{Relative moment errors, \NEMO\ ($\negat=1$): $\mom_0$ (solid line), $\mom_{1/2}$ (dash-dotted line), $\mom_{1}$ (cross) and $\mom_{3/2}$ (circle).} 
\label{fig:err_alg_1_nq2}
    \end{minipage}%
    \hspace{0.22cm}
    \begin{minipage}{0.48\textwidth}
         \includegraphics[width=1.0\linewidth]{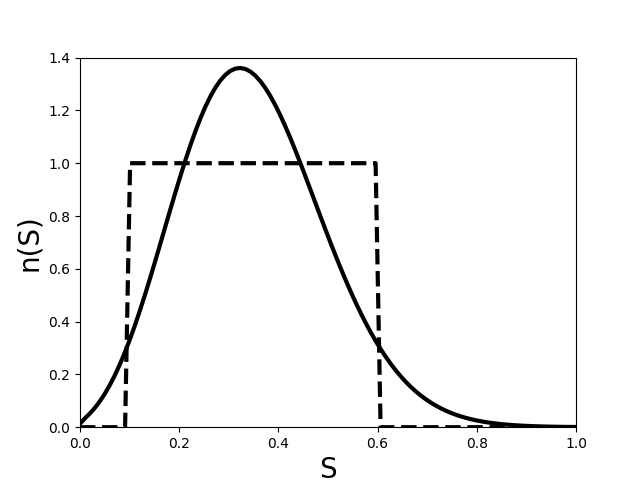}
   \caption{The ME-reconstructed NDF in the solid line and the initial discontinuous NDF in the dashed line.}
        \label{fig:initial_ndf_2}
    \end{minipage}
    \vspace{-0.5cm}
\end{figure}

In this section, we compare the solution obtained with \NEMO\ algorithm ($\negat=1$), using the time step $\deltat=0.002$, with the exact kinetic solution. One can see from Figure \ref{fig:ndf-1}, where we plot the reconstructed NDF of \NEMO\ solution and the exact kinetic distribution at $t=0.1$ and $t=0.2$, that the ME-reconstructed distribution follows accurately the exact distribution. It is important to notice that we are able to capture an accurate solution of the size distribution by using  only four moments. 
To go further in quantitative comparison, Figure \ref{fig:err_alg_1_nq2} shows the evolution of the relative error of the four fractional
 moments. The moment errors relatively to their initial values 
 do not exceed $0.2\%$. 
 \NEMO\ algorithm then allows to obtain an accurate solution of the moments by using only three quadrature points \eqref{algo:NEMO-step4}. This avoids to compute directly the integral in \eqref{eq:Kinetic-moment-sol} by using a large number of Gauss-Legendre quadrature points to approximate the moment-integrals. 
\subsubsection{$d^2$ law evaporation for a discontinuous initial NDF}
In this second case, we test the new algorithm \NEMO\ in the case of a discontinuous initial NDF:
\begin{equation}
    n^0(S)= 
\begin{cases}
   1,& \text{if } S \in [0.1,0.6]\\
    0,             & \text{otherwise}
\end{cases}
\label{eq:carre_initial_NDF}
\end{equation}
The initial NDF defined in \eqref{eq:carre_initial_NDF} and the initial ME-reconstructed NDF are plotted in Figure \ref{fig:initial_ndf_2}. In Figure \ref{fig:ndf-2}, we present the reconstructed NDF computed using \NEMO\ algorithm with $\negat=1$ at two different times. In the same figures, we compare the obtained results with the exact kinetic solution  \footnote{The exact kinetic solution is computed for the initial ME reconstructed size distribution.}. As in the previous case, \NEMO\ algorithm shows an accurate prediction of the exact kinetic solution. Furthermore, Figure \ref{fig:err-2}-left shows that the relative error of the moments are less than $1.\%$, in the case of $\negat=1$ and $\deltat=6.e-3$, which is an accurate result. We can improve the accuracy of \NEMO\ algorithm by using a smaller time step $\deltat=6.e-4$ or by using $\negat=2$ and keeping the same time step as the previous case. In figure \ref{fig:err-2}-right, we present the relative moment error using \NEMO\ algorithm with $\negat=2$ and $\deltat=6.e-3$. In this case, the relative 
moment errors are less than $0.3\%$.

\begin{figure}
\begin{subfigure}{.48\textwidth}
  \centering
  \includegraphics[width=.95\linewidth]{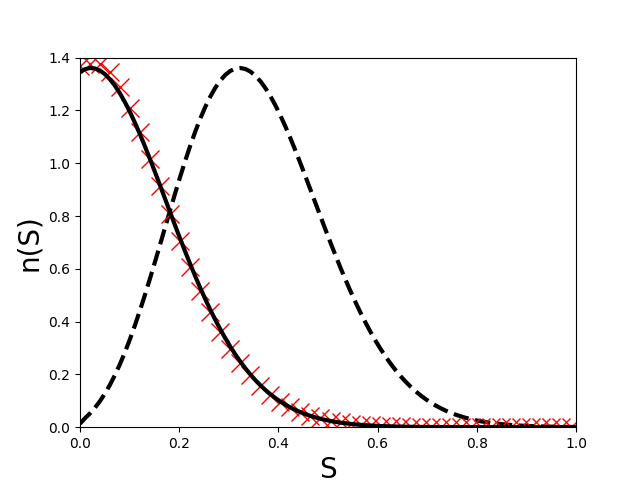}
  \label{fig:ndf_2_t02}
\end{subfigure}%
\begin{subfigure}{.48\textwidth}
  \includegraphics[width=.95\linewidth]{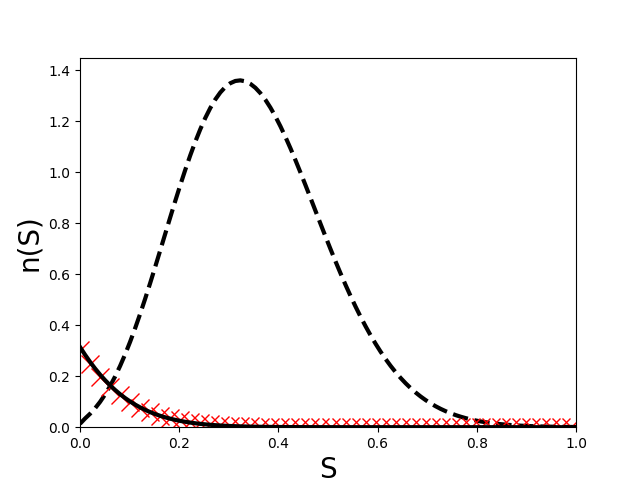}
  \label{fig:ndf_2_t04}
\end{subfigure}
 \caption{Solutions of ME-reconstructed NDF at $t=0.3$ (left) and $t=0.6$ (right) using: \NEMO\ with $\negat=1$ (cross), exact kinetic solution (solid line) and the initial distribution (dashed line).}
\label{fig:ndf-2}
\end{figure}

\begin{figure}
\begin{subfigure}{.48\textwidth}
  \centering
  \includegraphics[width=.95\linewidth]{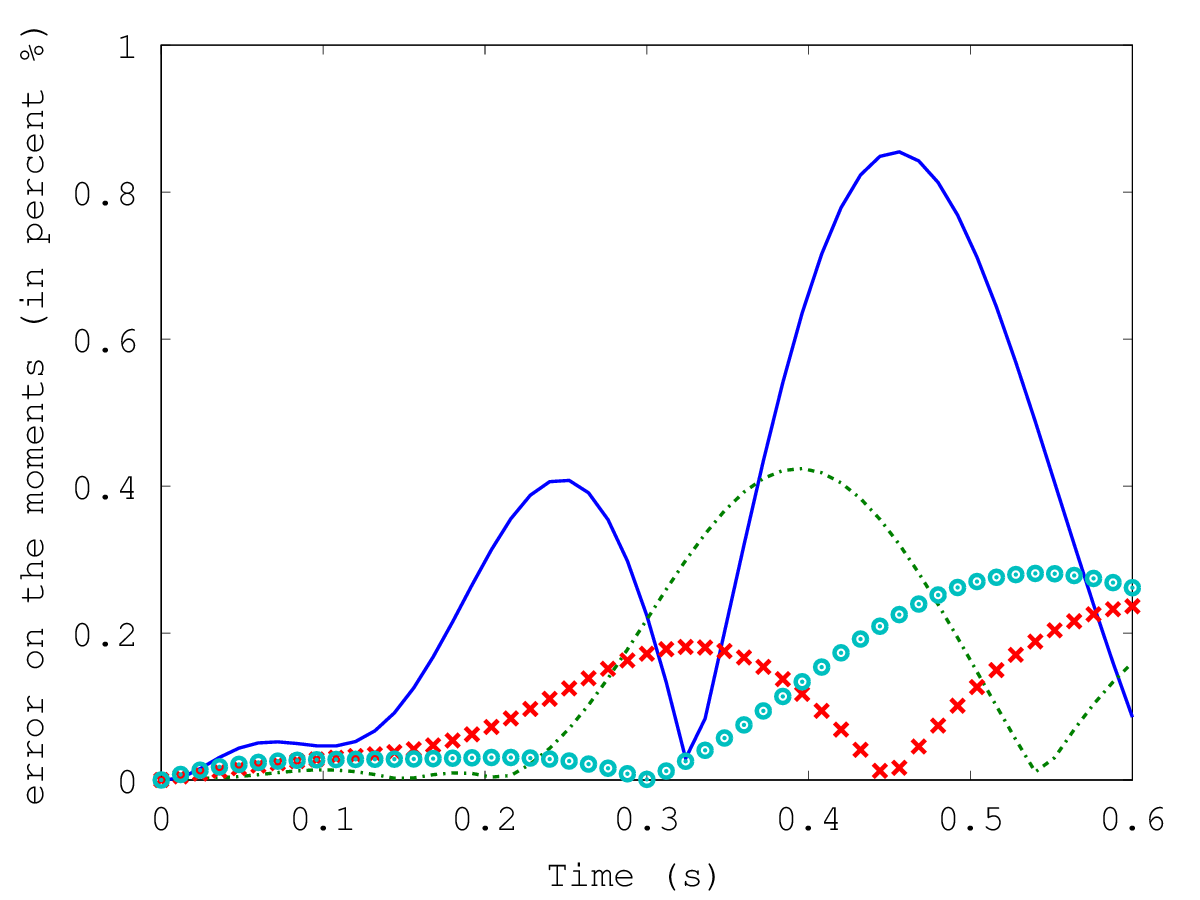}
  \label{fig:err_kin_2}
\end{subfigure}%
\begin{subfigure}{.48\textwidth}
 \includegraphics[width=.95\linewidth]{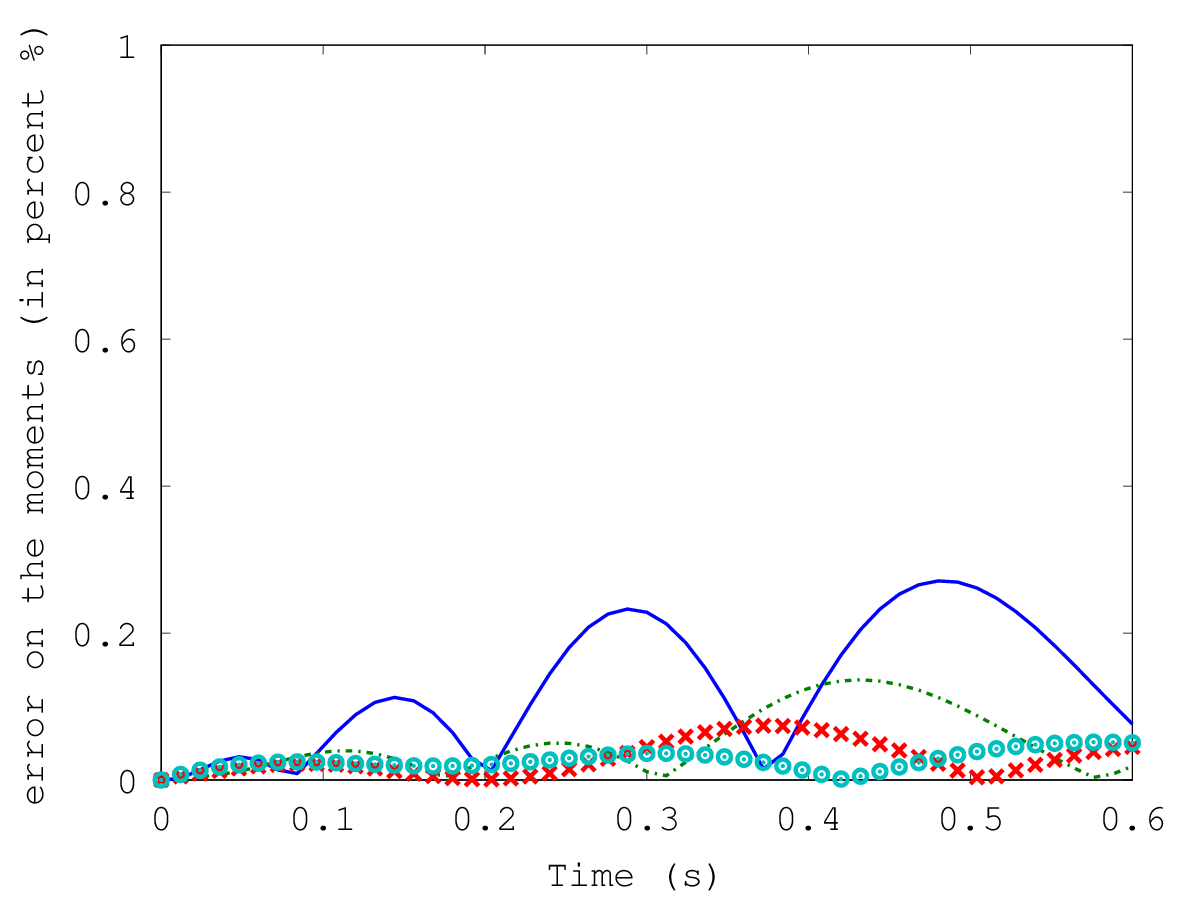}
  \label{fig:err_al_2}
\end{subfigure}
 \caption{Evolution of the relative moment errors using \NEMO\ with $\negat=1$ (left) and $\negat=2$ (right): $\mom_0$ (solid line), $\mom_{1/2}$ (dash-dotted line), $\mom_{1}$ (cross) and $\mom_{3/2}$ (circle).}
\label{fig:err-2}
\end{figure}

\subsection{Accuracy of \NEMO\ algorithm for non constant evaporation laws}
\label{app:linear-evaporation}
\NEMO\ scheme has been developed under the assumption of a $d^2$ law, but as it was explained before, the algorithm can be generalized for more complex laws by solving the Lagrangian equation \eqref{eq:Lagrang-size-evol} for each abscissas $\size_j$ given in the third step of the algorithm. In this section, we propose to evaluate the accuracy of the algorithm in the case where the evaporation rate depends linearly on the size:
\begin{equation}
R_s(S)=-(a+bS).
\end{equation}
The exact kinetic solution can be computed according to equation \eqref{eq:exact-kinetic-sol2}. In the following, we set $a=0.5$ and $b=1$. Figure \ref{fig:ndf-3} presents the ME-reconstructed NDF computed by \NEMO\ algorithm, 
and  the exact kinetic solution of the NDF at $t=0.3$ and $t=0.6$. The relative errors of the moments are given in Figure \ref{fig:err-3}. We can see from these results, the accuracy of the generalized \NEMO\ algorithm to predict the kinetic evolution.
\begin{figure}
\begin{subfigure}{.48\textwidth}
  \centering
  \includegraphics[width=.95\linewidth]{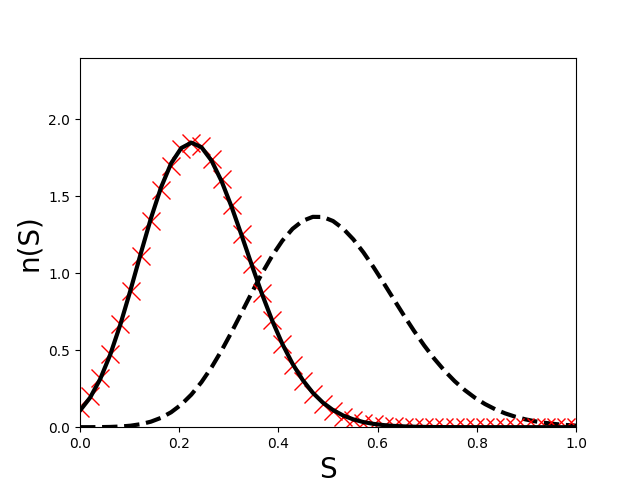}
  \caption{t=0.3} 
  \label{fig:ndf_3_t03}
\end{subfigure}%
\begin{subfigure}{.48\textwidth}
  \includegraphics[width=.95\linewidth]{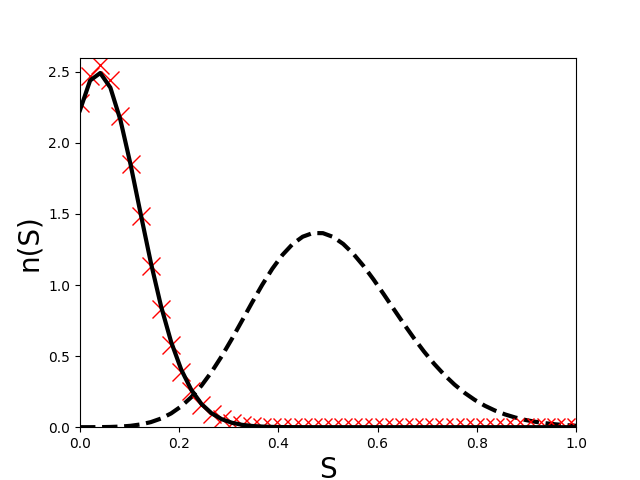}
  \caption{t=0.6} 
  \label{fig:ndf_3_t06}
\end{subfigure}
 \caption{The evolution of the NDF in the case of a linear evaporation rate: initial ME reconstructed solution (dashed line), \NEMO\ algorithm using $\negat=1$ (cross), and exact kinetic solution (solid line), at times $t=0.3$ and $t=0.6$.}
\label{fig:ndf-3}
\end{figure}

\begin{figure}[!htb]
    \vspace{-0.25cm}
    \begin{minipage}{.48\textwidth}
  \includegraphics[width=.98\linewidth]{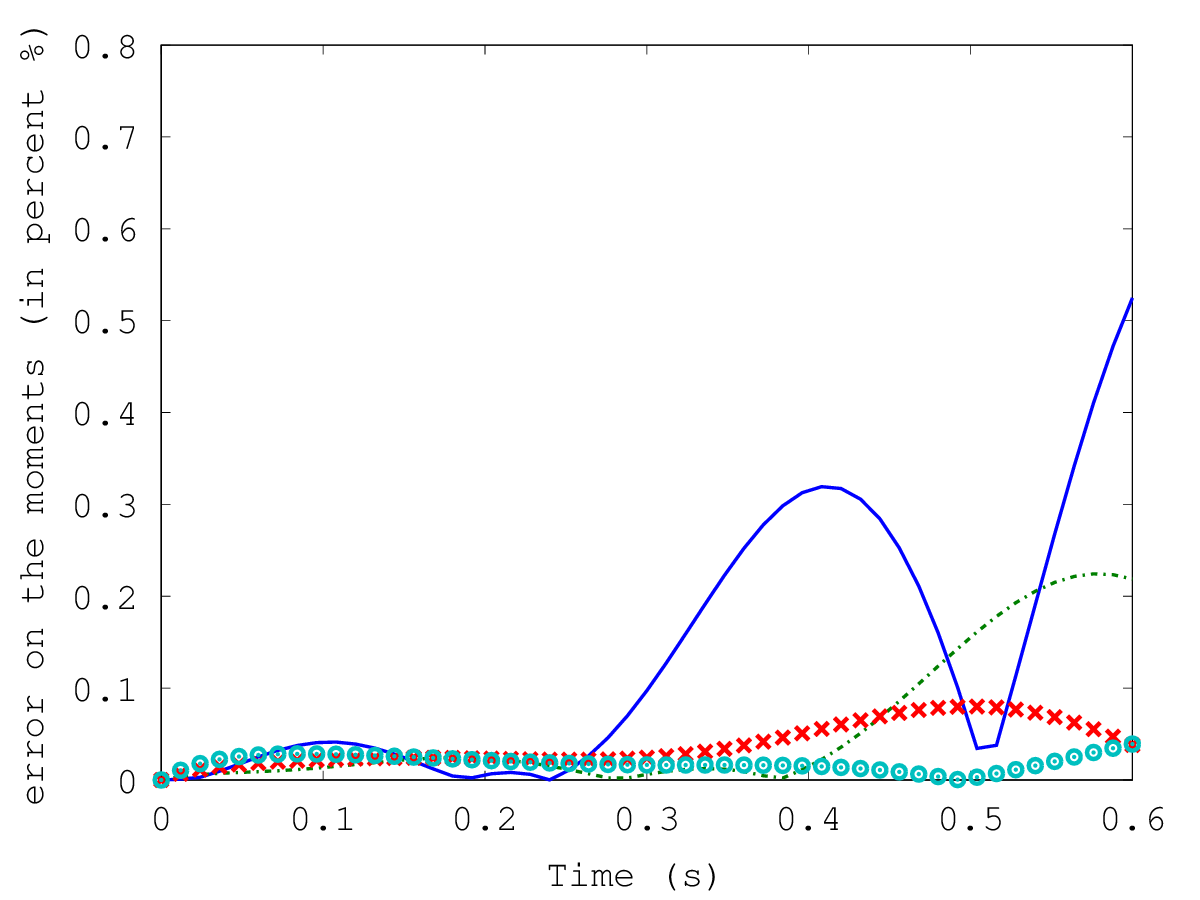}
 \caption{Evolution of the moment errors relatively to their initial value, \NEMO\ algorithm: $\mom_0$ (solid line), $\mom_{1/2}$ (dash-dotted line), $\mom_{1}$ (cross) and $\mom_{3/2}$ (circle).} 
\label{fig:err-3}
    \end{minipage}%
    \hspace{0.22cm}
    \begin{minipage}{0.48\textwidth}
         \includegraphics[width=1.0\linewidth]{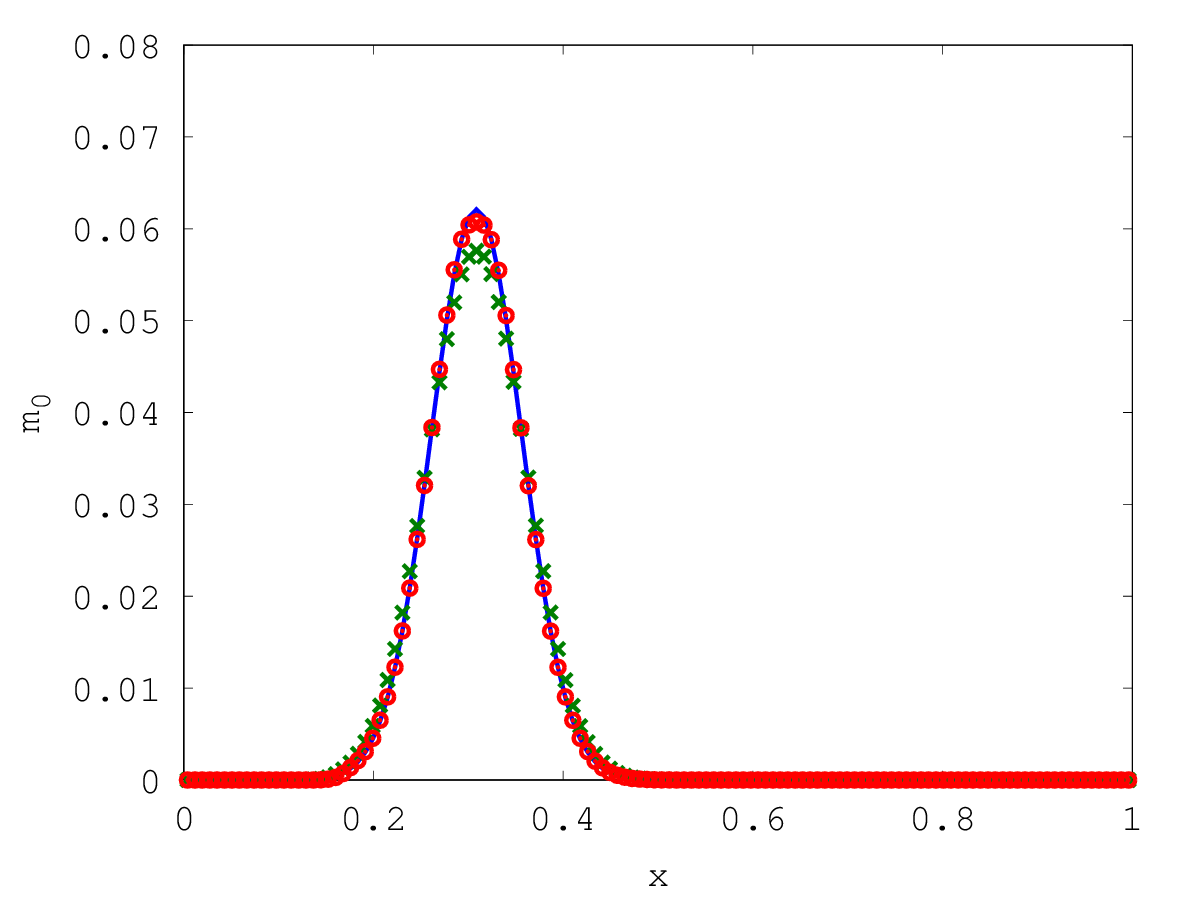}
   \caption{Spatial profile of the moment $\mom_0$ at $t=0.5$ with $128$ grid cells. First order scheme (cross), second order scheme (circle) and exact solution (solid-line).}
        \label{fig:1D_m0}
    \end{minipage}
    \vspace{-0.5cm}
\end{figure}

\subsection{Transport in 1D simulation}
In this part, we investigate the accuracy and robustness of the transport schemes, the first order and the second order detailed in the Appendix \ref{app:Transport-scheme}. The evaporation and the drag force are not considered in this study. First, we investigate the order of accuracy of the transport schemes. Then, we test the two numerical schemes in a critical case, where we generate a $\delta$-shock singularity.
\subsubsection{Accuracy order study}
We consider an initial size distribution with a form depending on the coordinate space $x$. In fact, we need a non trivial initial size-space distribution, such that the slopes used in the second order scheme do not vanish all the time. The chosen initial size distribution has the following profile:
\begin{equation}
n(t,S,x)=10\exp\left(-\frac{(x-x_c)^2}{\sigma_x^2}\right)\exp\left(-\frac{(\sqrt{S}-(1-x)/2)^2}{\sigma_R}\right),
\end{equation}
where $x_c=0.25$, $\sigma_x=0.1$ and $\sigma_R=0.3$. The initial velocity field is initiated as follows:
\begin{equation}
\left\{
\begin{array}{rclc}
u(t=0,x) &=& 0.5-x & x<0.5,\\
u(t=0,x) &=& 0. & x\geq 0.5.
\end{array}
\right.
\end{equation}

Figure \ref{fig:1D_m0} shows the spatial distribution of the moment $\mom_0$ obtained with the first and second order schemes at $t=0.5$. Comparing the two schemes, it is clear that the results obtained with the second order scheme are more accurate than the ones obtained with the first order. In Figure \ref{fig:1D_convergence}, we show the grid convergence for the first and second order to the analytic solution at $t=0.8$. The second order scheme converges faster than the first order scheme to the solution. In the extremum of the moment spatial profile, we can see that the convergence to solution is slow compared to the other points. In fact, the slope limitation in this point is activated to ensure a non oscillating solution, but, it introduces some numerical diffusion.

\begin{figure}[!htb]
 \vspace{-0.5cm}
    \centering
    \begin{minipage}{.5\textwidth}
        \centering
       \includegraphics[width=1.\linewidth]{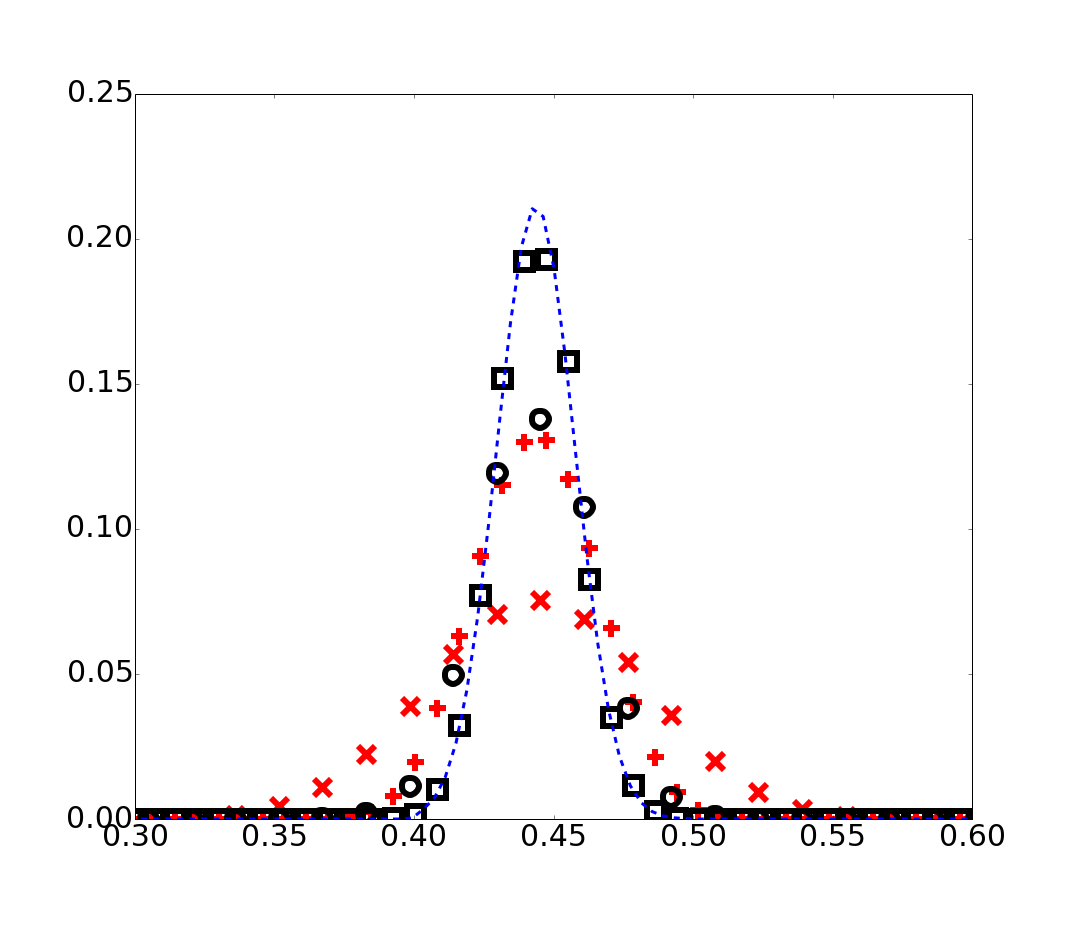}
       \caption{Grid convergence; first order: $64$ cells (cross) and $256$ cells (plus), second order: $64$ cells (circle) and $256$ cells (square). Exact solution in blue dashed-line.}
\label{fig:1D_convergence}
    \end{minipage}%
     \hspace{0.15cm} 
    \begin{minipage}{0.47\textwidth}
        \centering
        \includegraphics[width=1.\linewidth]{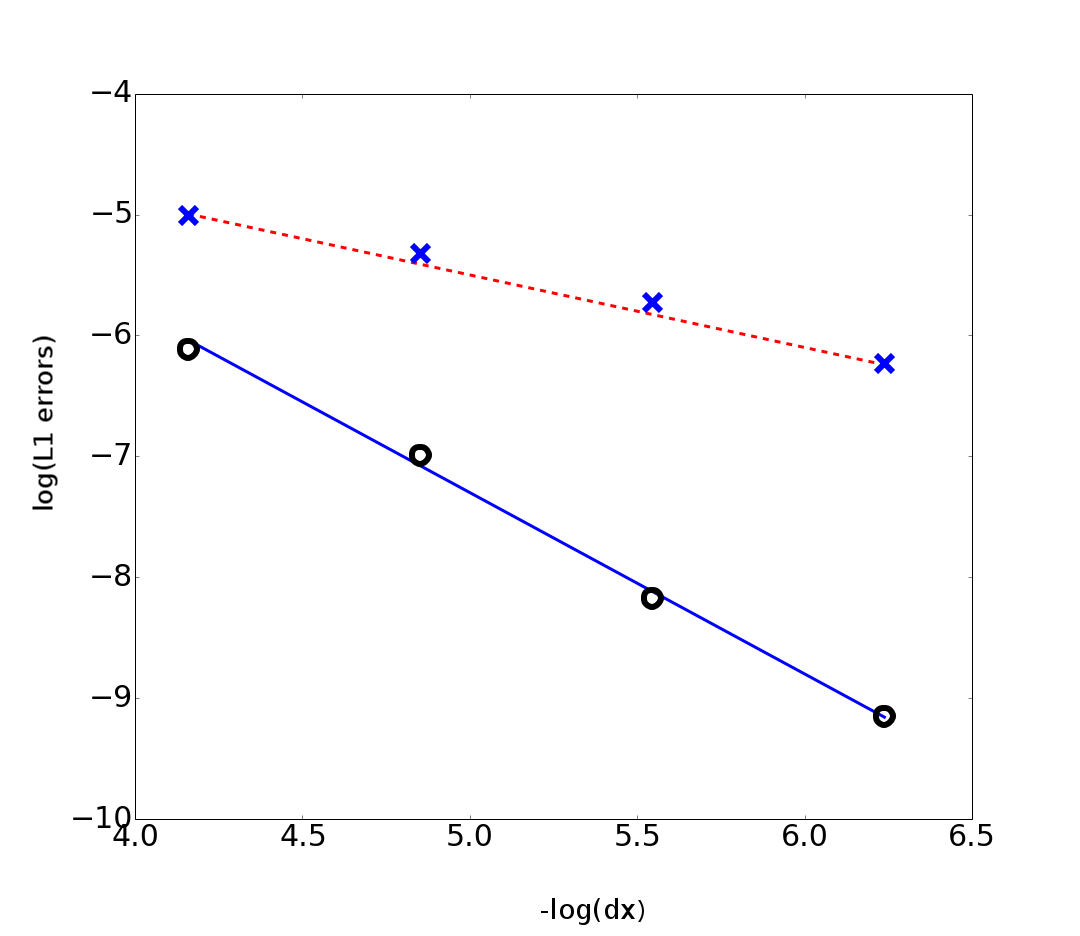}
      \caption{Error curves of $m_0$ with respect to grid refinement in logarithm scale: first order (cross) and second order (circle) schemes.}
        \label{fig:order-curve}
    \end{minipage}
    \vspace{-0.4cm}
\end{figure}

We also compute the $L^1$-error for each numerical scheme depending on the grid size. In Figure \ref{fig:order-curve}, we display the order curve for the two schemes. As it can be seen from these curves, the order of the first order scheme is around $0.6$ and of the second order, around $1.5$. We mention that the order was computed in a critical case. Indeed, the characteristic curves cross at $t=1.0$ and the order of accuracy was evaluated at $t=0.8$. At this time the solution starts to have high gradient variations, thus the observed decrease of the order of accuracy.

\subsubsection{Robustness and capacity of capturing $\delta$-shocks}
The monokinetic assumption is not a valid hypothesis when droplets cross. In such an event, the monokinetic model generates a $\delta$-shock. Despite the non-physical solution, the two kinetic schemes should be able to provide a solution in all critical situations and even resolve the created singularities. In this section, we test the robustness of the numerical schemes in a case of a strong crossing. Initially, we consider two spatial Gaussian distributions centered at two positions and traveling at opposite velocity as shown in Figure \ref{fig:initial_shock}.
\begin{figure} 
\vspace{-0.3cm}
\begin{subfigure}{.48\textwidth}
  \centering
  \includegraphics[width=1.\linewidth]{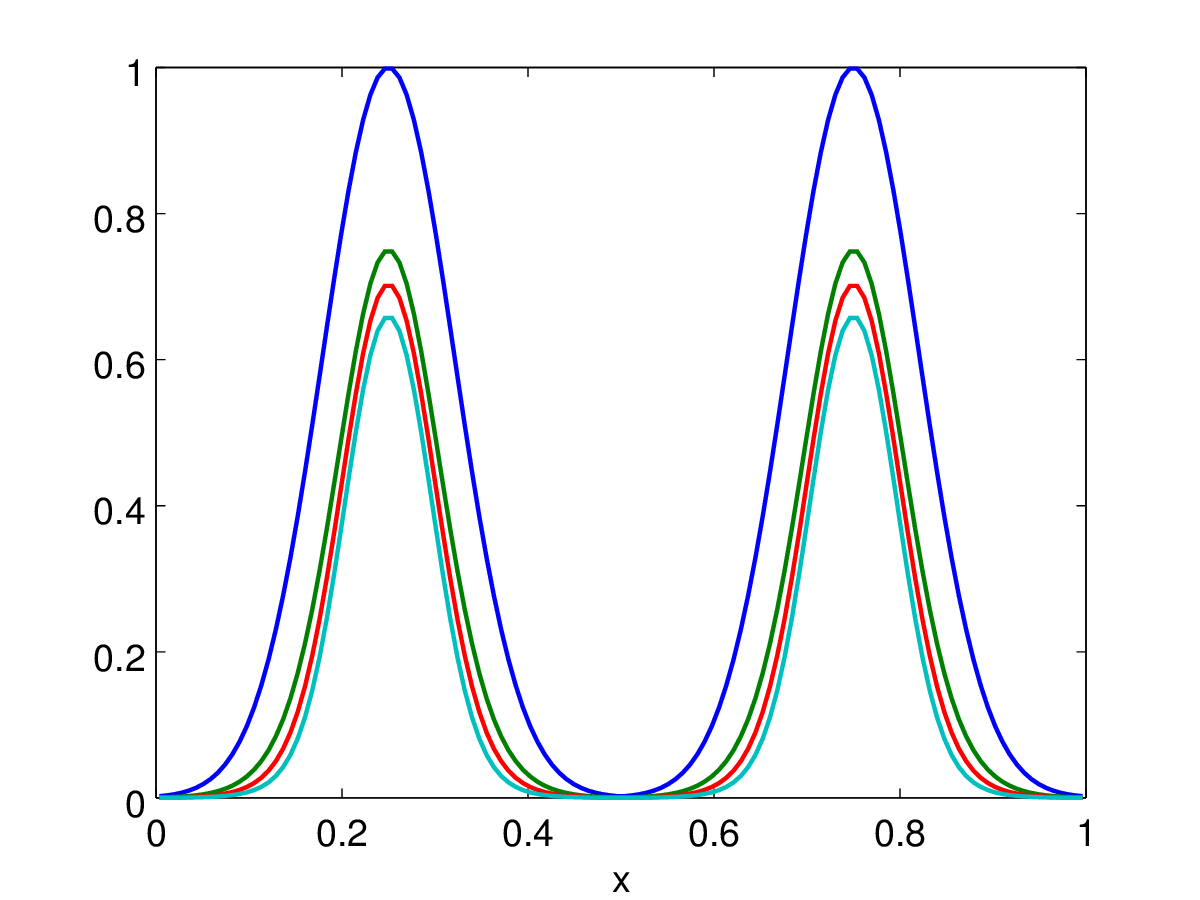}
\end{subfigure}%
\begin{subfigure}{.48\textwidth}
  \includegraphics[width=1.\linewidth]{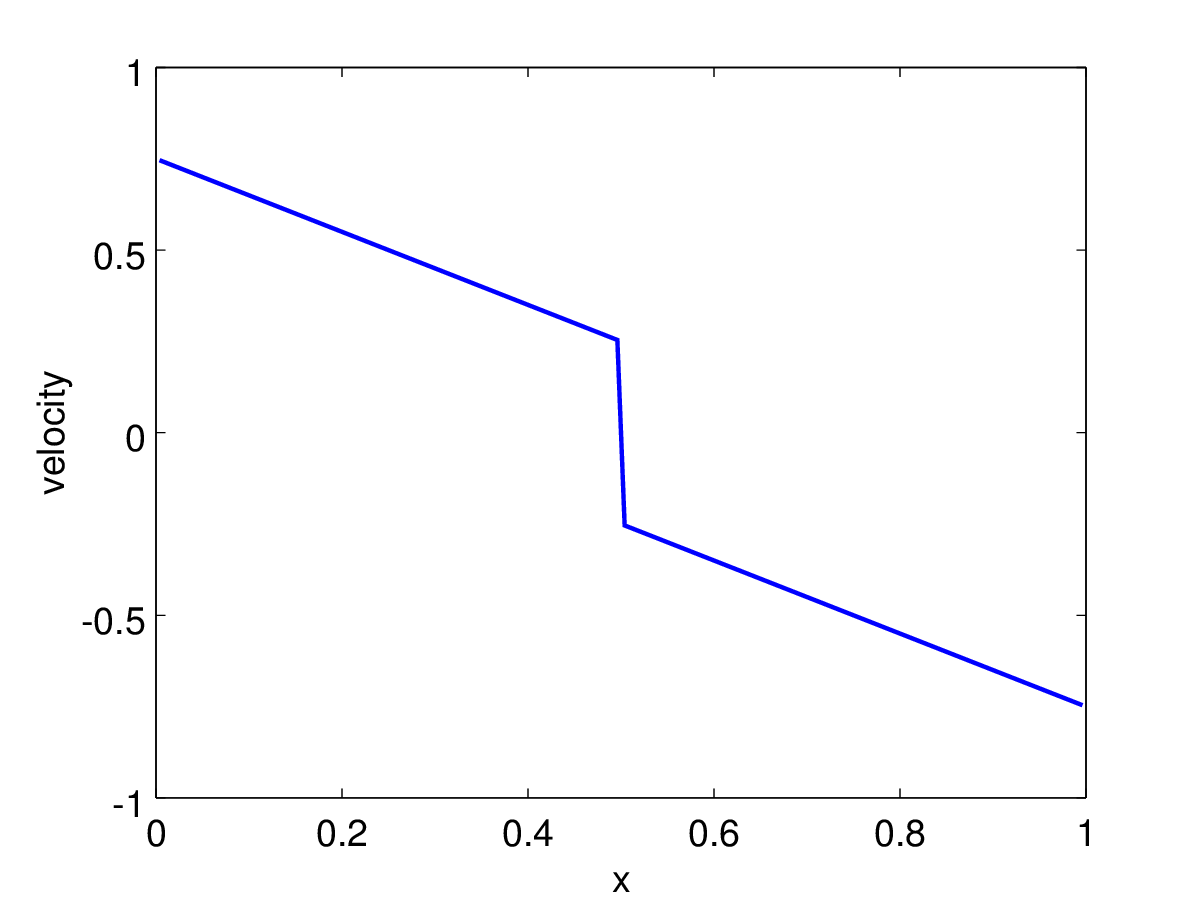}
\end{subfigure}
 \caption{Initial condition for the crossing case. Left: Initial moment fields, the curves represent the moment with decreasing order in terms of value. Right: initial velocity.}
 \vspace{-0.5cm}
\label{fig:initial_shock}
\end{figure}
Figure \ref{fig:shock} shows the spatial profile of the moment $\mom_0$ at two different times. We can see the shock generation when the two packets cross. At the end, the droplets accumulate at the center and the corresponding spatial distribution has the form of a  Dirac delta measure. The computation is still robust in this case and realizability preserved.

\begin{figure}
\begin{subfigure}{.48\textwidth}
  \centering
  \includegraphics[width=.9\linewidth]{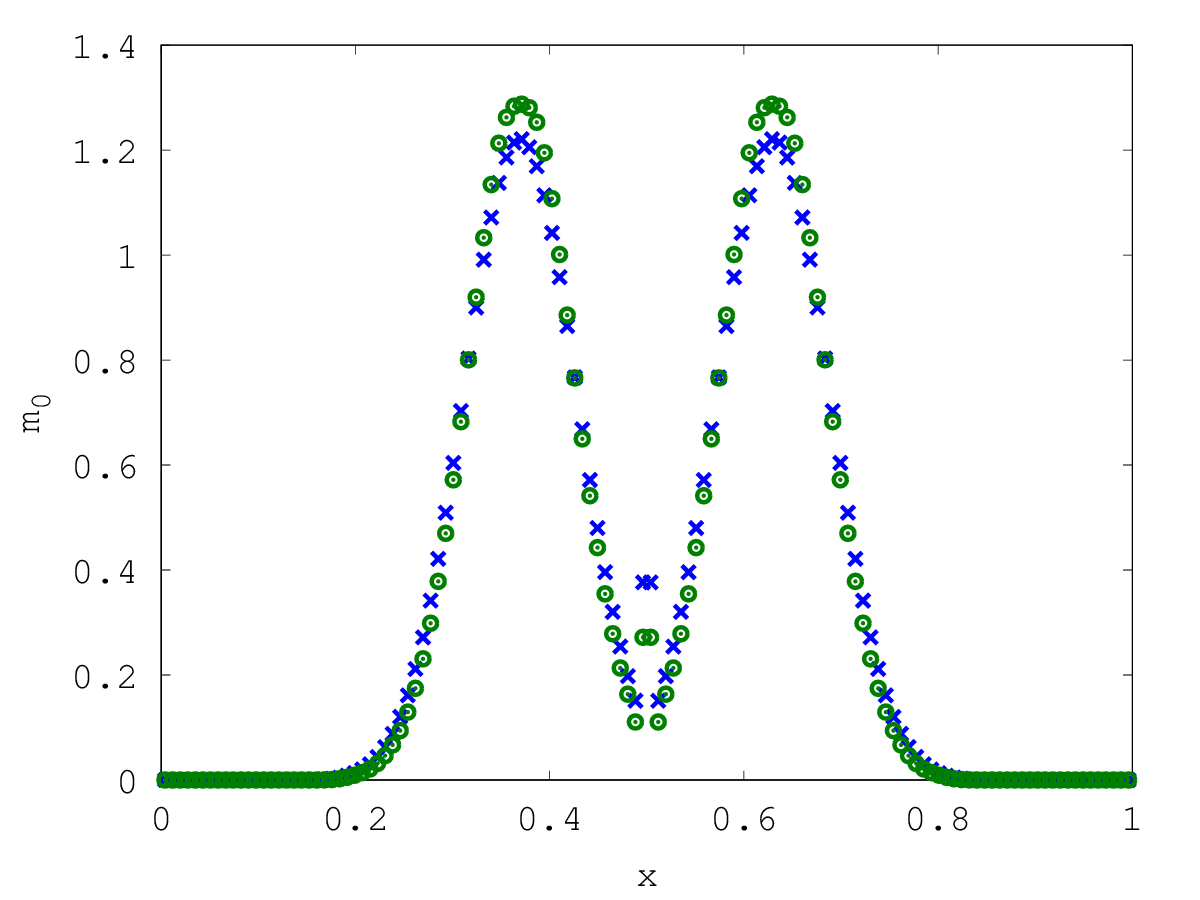}
\end{subfigure}%
\begin{subfigure}{.48\textwidth}
  \includegraphics[width=.9\linewidth]{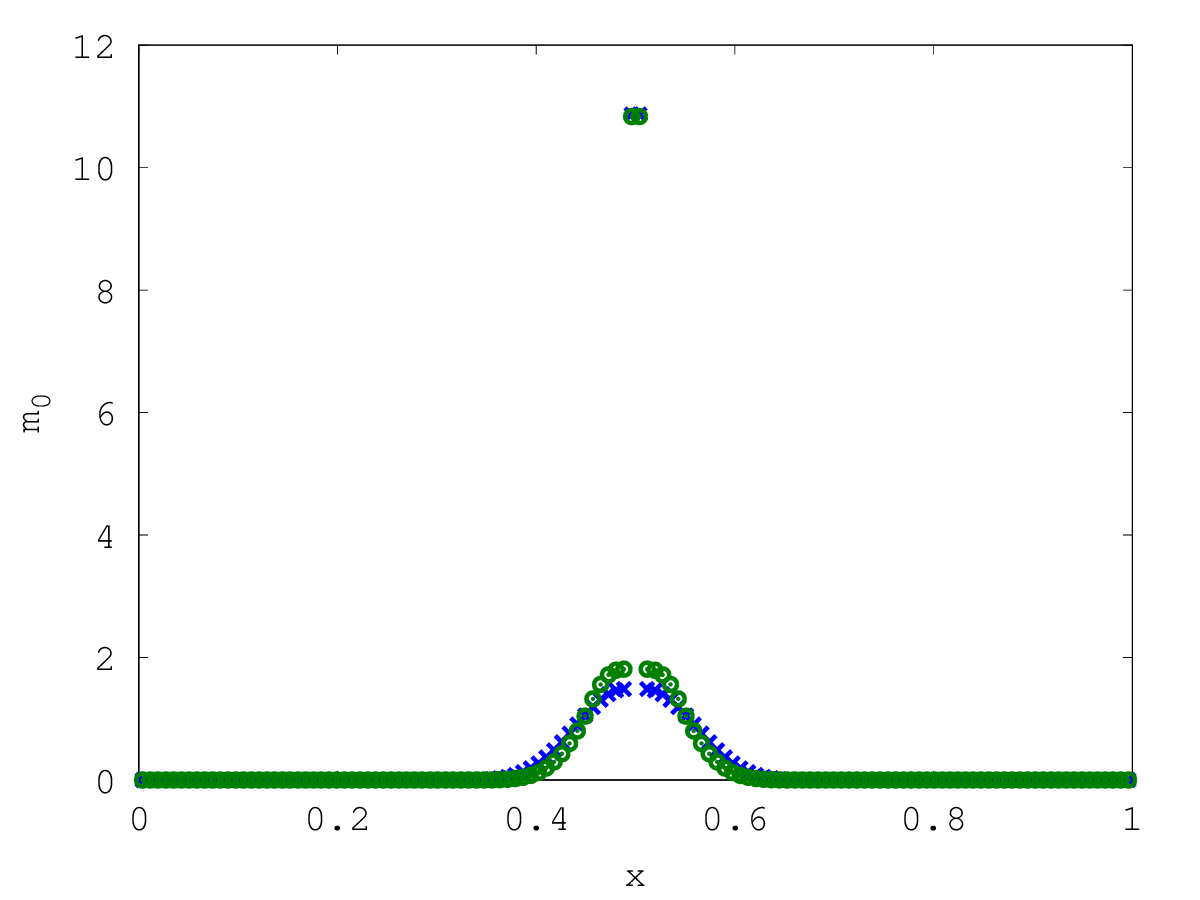}
\end{subfigure}
 \caption{The spatial profile of $\mom_0$ in the crossing case at $t=0.6$ (left) and $t=1.2$ (right), using $128$ cells: first order scheme (cross) and second order scheme (circle).}
 \vspace{-0.5cm}
\label{fig:shock}
\end{figure}

\subsection{2D simulation: Transport, evaporation and the drag force}
After the model and numerical schemes have been tested in 0D and 1D simulations, we propose in this part to compare in a classical 2D configuration the present model and numerical schemes with the EMSM model and its original numerical schemes \cite{massot2010,kah12}, in a case where we consider transport, evaporation and drag. The simulations are performed using the \canop\ code, developed within the collaboration of Maison de la Simulation, IFPEn and EM2C Laboratory. It is based on the 
 \pf\ library \cite{BursteddeWilcoxGhattas11}, a library providing Adaptive Mesh Refinement capability highly scalable in massively parallel computations \cite{drui_PhD,OGST}. 
 In the present simulation, we use only a uniform grid, since we are not concerned with Adaptive Mesh Refinement in the present work. We consider an evaporating spray in the presence of Taylor-Green vortices for the gas, which is a steady solution of the inviscid incompressible Euler equations. The non-dimensional velocity field of the gas reads:
\begin{equation}
u_g(x,y)=\sin(2\pi x)\cos(2\pi y),\qquad
v_g(x,y)=-\cos(2\pi x)\sin(2\pi y),
\end{equation}
where $(x,y)\in[0,1]^2$ and with periodic boundary conditions. Initially, the spray is localized in the bottom-left vortex. The initial spatial size-distribution is given by:
\begin{equation}
\ns(t,\xv,\size)=\mathds{1}_{[a,b]}(\size)\mathds{1}_{\left\{\xv^{\prime},||\xv^{\prime}-\xv_c||_2<\sqrt{2}r\right\}}(\xv)\exp(-||\xv-\xv_c||_2^2/r^2),
\end{equation}
where $[a,b]=[0.25,0.75]$, $\xv_c=(0.15,0.15)$ and $r=0.1$. The initial Stokes number computed according to the mean size $\bar{\size}=\mom_1/\mom_0$ is $\Stokes(\bar{\size})=0.05$, which is close to the critical Stokes number. The Stokes number decreases over time because of evaporation. The spray evaporation rate is $\KEvap=0.5$. 
Figures \ref{fig:EMSMvsFRAC_t05}-\ref{fig:EMSMvsFRAC_t1} present the computed spatial distribution of the volume fraction at two different times, using the EMSM model (left) and the new fractional moment model (right). 
We have used second order scheme for the transport resolution in both cases (see \cite{kah12} for EMSM and \ref{app:Transport-scheme} for the fractional moment model) and the original evaporation algorithm \cite{massot2010} to solve evaporation in the EMSM model and \NEMO\ algorithm with $\negat=1$ for fractional moments. 
For the EMSM model, the volume fraction is not resolved but calculated through ME reconstruction of the size distribution $(1/6\sqrt{\pi})\int_0^1\size^{3/2}\nME(\size)\mathrm{d}\size$. Instead, for the new model, the volume fraction is directly calculated as $(1/6\sqrt{\pi})\mom_{3/2}$. The results of the two computations are closely similar, and the $L^1$-norm difference relatively to the initial volume fraction field is less than $3\%$ at $t=1$. 

Figure \ref{fig:plotoverline} provides a quantitative comparison, where we plot the volume fraction along the segments: 
$\left(x=0.3,\,\,y\in[0,0.5]\right)$ and $\left(x\in[0,0.5],\,\,y=0.3\right)$ at $t=0.5$. 
Let us underline that the EMSM model in this configuration was validated through a detailed comparison with the Multi-fluid model in \cite{kah12} and 
shown to be able to predict the evaporation and the mean dynamics of the spray\footnote{EMSM model and present model, with one size-section, are limited in predicting the size-conditioned dynamics compared with Multi-fluid model with a Stokes number computed from a mean size. The  Correlated Size Velocity Moment model \cite{vieJCP2012} allows to tackle this issue.}. Besides, the computational time was shown to be much shorter for the EMSM model, by a factor almost related to the ratio of the number of PDE solved in the model\footnote{Switching from 30 variables - multi-fluid -  to 6 variables - EMSM -  in 2D resulted in a acceleration factor of more than 4, the maximization of entropy being responsible for not reaching a factor of 5.}.

\begin{figure}
\begin{subfigure}{.48\textwidth}
  \centering
  \includegraphics[width=.84\linewidth]{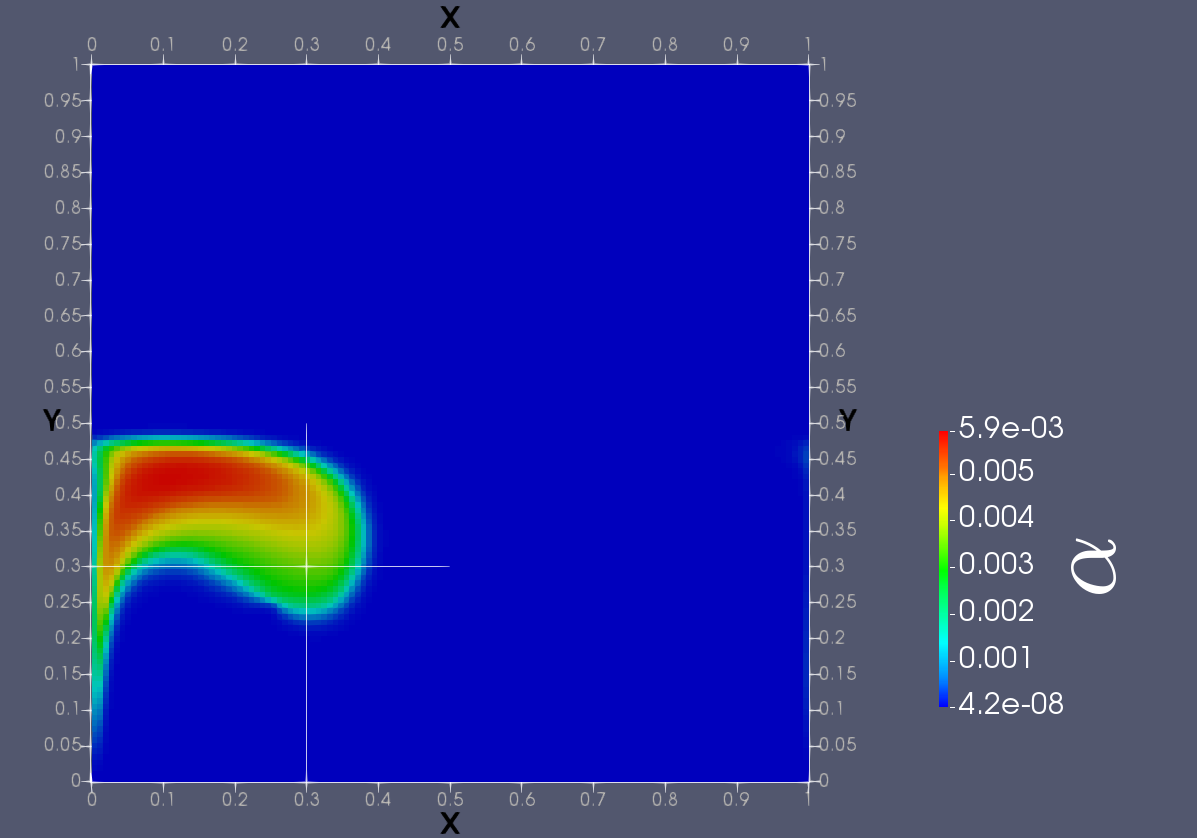}
  \caption{ EMSM } 
  \label{fig:emsm_t05}
\end{subfigure}%
\begin{subfigure}{.48\textwidth}
  \includegraphics[width=.84\linewidth]{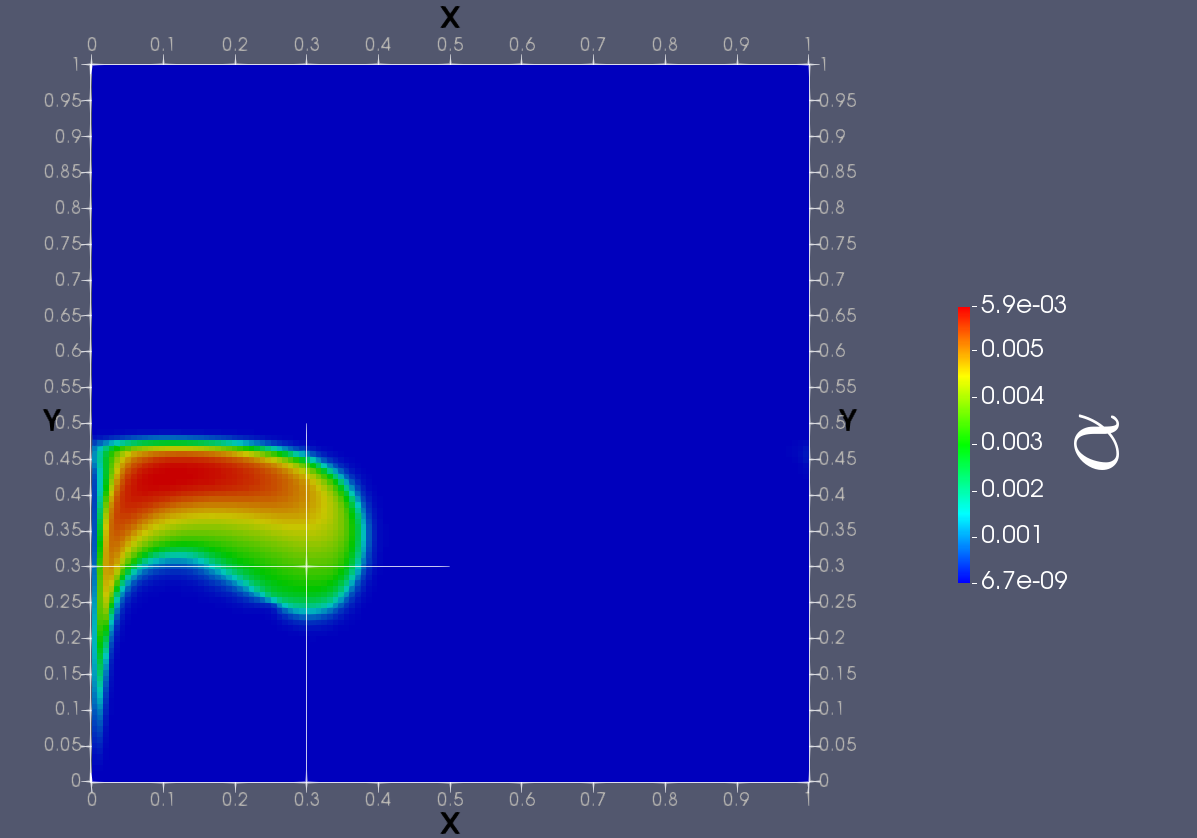}
  \caption{Fractional moments} 
  \label{fig:frac_t05}
\end{subfigure}
 \caption{The spatial distribution of the volume fraction for the Taylor-Green simulation at $t=0.5$. The computation is carried out in a uniform grid $128\times128$.}
\label{fig:EMSMvsFRAC_t05}
 \vspace{-0.5cm}
\end{figure}
\begin{figure}
\begin{subfigure}{.48\textwidth}
  \centering
  \includegraphics[width=.86\linewidth]{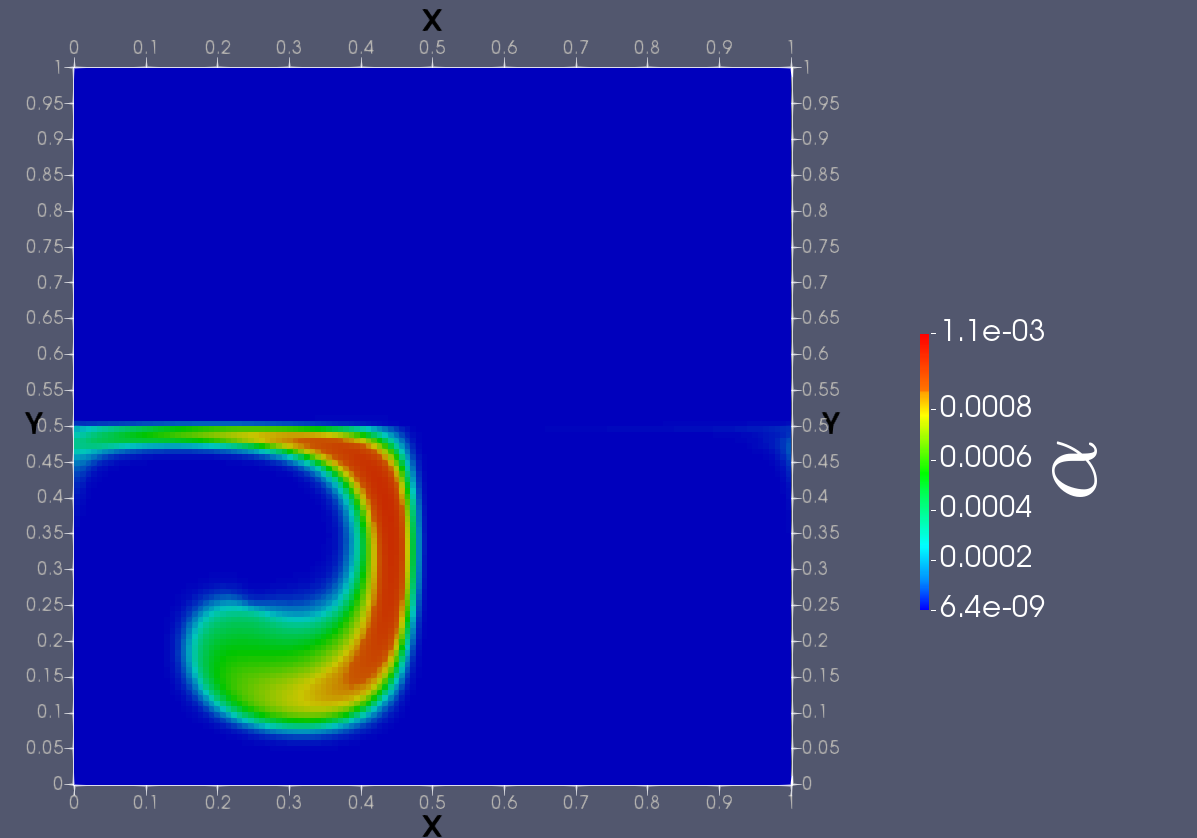}
  \caption{ EMSM } 
  \label{fig:emsm_t1}
\end{subfigure}%
\begin{subfigure}{.48\textwidth}
  \includegraphics[width=.86\linewidth]{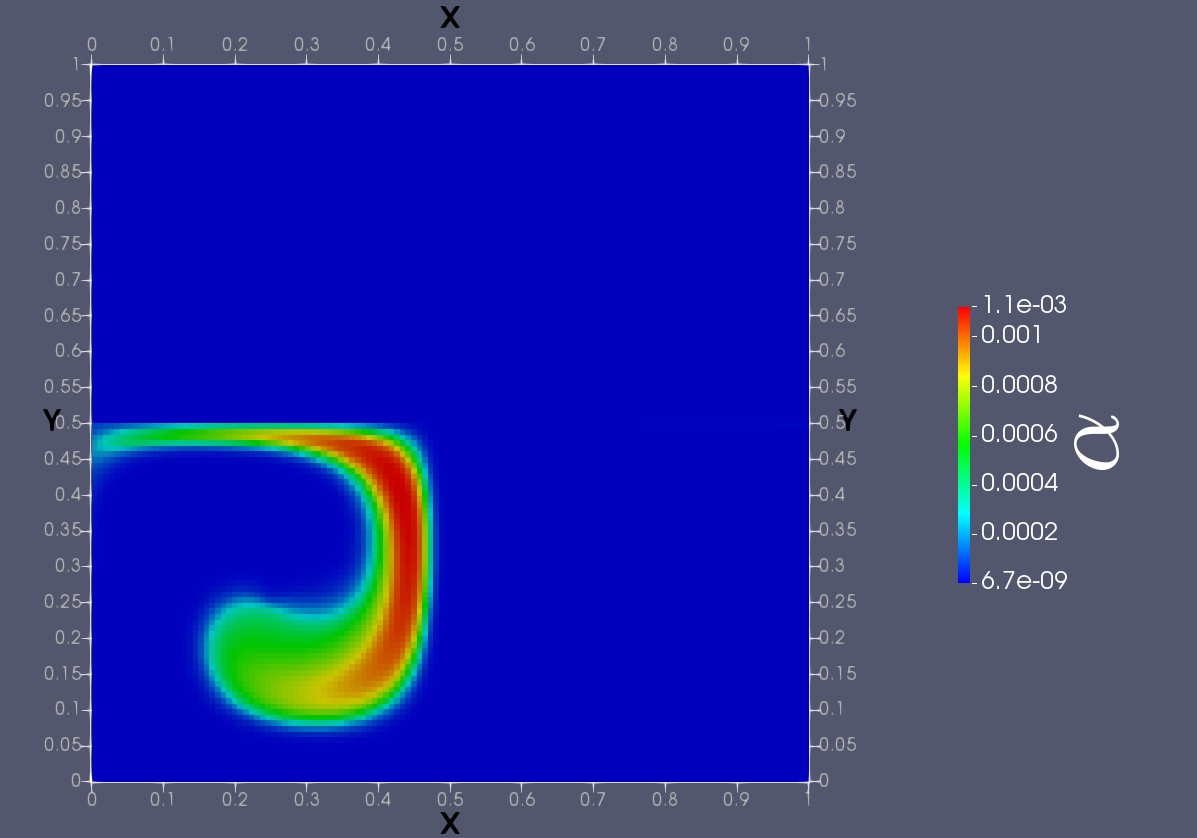}
  \caption{Fractional moments} 
  \label{fig:frac_t1}
\end{subfigure}
 \caption{The spatial distribution of the volume fraction for the Taylor-Green simulation at $t=1.0$. The computation is carried out in a uniform grid $128\times128$.}
\label{fig:EMSMvsFRAC_t1}
 \vspace{-0.6cm}
\end{figure}

Thus we can conclude on two points: 1- the comparison of the two models validates the results of the new model and assesses its ability to capture the dynamics and evaporation of the spray. 2- The new model uses the same number of variables as in the EMSM model, but the computational time in this particular case is 1.8 times faster. In fact the EMSM model requires a supplementary ME-reconstruction  of the NDF to compute the  volume fraction, which is one of the variables in the new model.  Even if we do not claim this number is generic, the new model is however going to be even cheaper in terms of computational time. 
Let us also emphasize that in our numerical strategy, even if a large amount of time can be spent in evaluating the source terms (including maximization of entropy) and evolution in phase space, operator splitting yields an embarrassingly parallel problem for which work stealing and task-based optimization can be extremely efficient.

\begin{figure}
\begin{subfigure}{.48\textwidth}
  \centering
  \includegraphics[width=.9\linewidth]{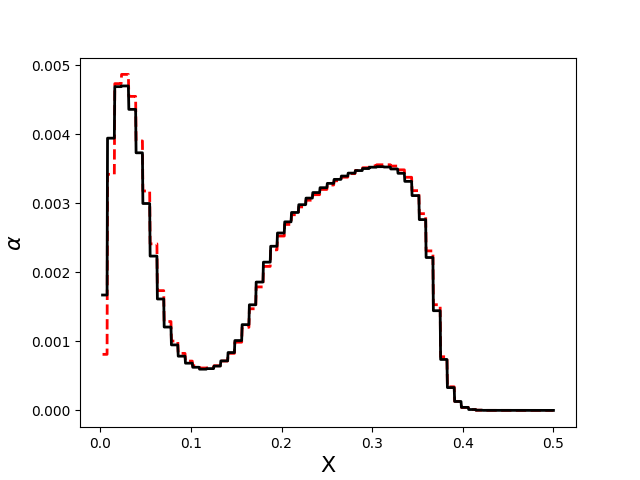}
\end{subfigure}%
\begin{subfigure}{.48\textwidth}
  \includegraphics[width=.9\linewidth]{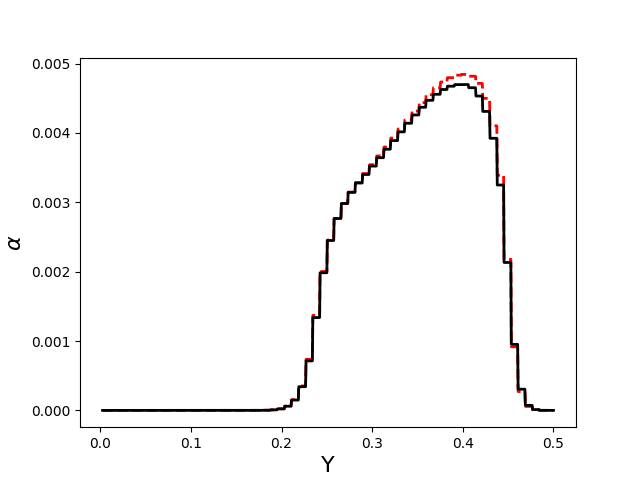}
\end{subfigure}
 \caption{Volume fraction profiles at $t=0.5$ with EMSM (dashed red line) and fractional moment model (solid black line) over the segments: $x\in[0.,0.5]$ and $y=0.3$ (left) and $x=0.3$ and $y\in[0.,0.5]$. }
\label{fig:plotoverline}
\end{figure}


\section{Conclusions}
\label{sec:conclusions}

In the present paper, we have proposed a novel Eulerian high order size moment model and numerical methods in order to describe 
polydisperse evaporating sprays, with the same accuracy and robustness based on built-in realizability preserving numerical schemes as previously developed methods. Its main originality in terms of modeling is to involve a set of variables, which can be interpreted physically as interface geometry quantities and are thus easier to couple to other zones of the flow, where the spray is denser or even in the separated-phases region.
Since the new model involves fractional order moments, we have 
 tackled the various mathematical difficulties related to  ME reconstruction and moment space theory \cite{dette97,lasserre2010} and proposed a new precise and robust numerical method, within the framework of operator splitting techniques \cite{doisneau14,descombes14} to separate the resolution of transport in physical space and phase space, for a model involving high order fractional size-moments as well as negative order moments. The way of treating negative order moments properly should be useful for various other problems in the several communities \cite{Frenklach2002}. The main difficulty is to design accurate, robust and realizable numerical schemes, especially for the resolution of the evaporation and transport, where standard solvers fail. 
 
Whereas the basis of the approach has been presented in a simplified modeling framework in order to focus on its applied mathematical aspects and for the sake of clarity, the extension to realistic models and configurations can be achieved. Detailed droplet models can be used in order to cope with predictive simulations, as well as two-way coupling and mesh motion, in internal combustion engines for IFPEn \cite{kah2015}. Higher order velocity moment methods, conditioned on droplet size and based on the Levermore hierarchy of hyperbolic system of conservation laws with an entropic structure \cite{levermore95,levermore96,sabat2016} allow to cope with turbulence and particle trajectory crossing. Implementation in the massively parallel code able to handle rather complex geometries has been accomplished \cite{drui_PhD} and the use of the present model assessed in a companion paper  \cite{OGST}, which rather focuses on implementation issues, parallelism and results. It shows the potential of the present model for realistic engine simulations.
 
However, the main challenge still remains the elaboration and design of a unified model for the liquid injection able to predict atomization as well as the generated polydisperse spray and its dynamics and evaporation \cite{essadki2018}, for which the present contribution offers some key ingredients.


\section*{Acknowledgments}
This research was supported by a PhD grant for M.
Essadki from IFPEn, as well as by Ecole Doctorale de Math\'ematiques Hadamard and EM2C. We want to thank A. Larat and A. Vi\'e for several interesting and helpful discussions. 

\appendix

\section{Introduction}

In the main paper, a set of fractional moments are used to describe polydisperse evaporating sprays and correspond 
to physically identified interface geometrical quantities.
In this Appendix, we present some mathematical elements related to the use of these fractional moments.
These elements are already known for integer moments and can be extended  in our case.
First, the fractional moment space is defined and the notion of canonical moments and quadrature is directly extended from the integer moments case.
Then, a complete proof is given for the existence and uniqueness of the Maximum Entropy reconstruction from any moment vector in the interior of the fractional moment space, as well as an algorithm to compute the parameters of this reconstruction.
Finally, using the canonical moments, the extension of the realizable transport scheme developed for integer moments in \cite{kah12} is provided.

\section{Fractional moment space, canonical moments and quadrature}
\label{appendix:fractional_moment_space}
The purpose of this part is to extend some definition and relevant properties of the integer moment space to fractional moment space, in order to provide necessary tools to characterize the geometry of this space and thus to design realizable numerical schemes. 
In the following, we use the normalized fractional moments $\cmom_{k/2}=\mom_{k/2}/\mom_0$.

\subsection{Fractional moment space - link with the integer moment space}
\begin{definition}
We define the Nth fractional moment space as follows:
\begin{equation*}
\Momsp{N}([0,1])=\left\{\CMom_N(\mu), \mu \in \Pset([0,1])\right\}, \hspace{0.1cm} \CMom_N(\mu)=(\cmom_0(\mu),\cmom_{1/2}(\mu),\hdots,\cmom_{N/2}(\mu))^t,
\end{equation*}
where $\Pset([0,1])$ denotes the set of all probability density measures defined on the interval $[0,1]$ and $\cmom_{k/2}=\int_0^1x^{k/2}d\mu(x)$.
\end{definition}

The fractional moments can be expressed as integer moments by using the following mapping $r^2=x$ for $x\in[0,1]$.
Let us then denote $\widetilde{\mu}$ the measure on the interval $[0,1]$, which is in a one to one correspondence with the measure $\mu$ on $[0,1]$ through the above transformation, that is: $\widetilde{\mu}([0,r])=\mu([0,r^2])$.
The moments can then be written:
\begin{equation}
\cmom_{k/2}(\mu)=\int_0^1x^{k/2}d\mu(x)
=\int_{0}^{1} r^kd\widetilde{\mu}(r).
\label{identif:integer-frac-mom}
\end{equation}
This relation expresses an identification between the fractional moment $\cmom_{k/2}(\mu)$ of the measure $\mu$ and the integer moment
$\widetilde{\cmom}_k=\cmom_k(\widetilde{\mu})$ of the measure $\widetilde{\mu}$. 
In the following, we use these notations to differentiate between the two natures of the moments, even if they are equal. With this simple identification, we will take benefit from the already existing results of integer moment space to extend them to the case of fractional moment space.

\subsection{Canonical moments}

For a fractional moment vector $\CMom_N\in\Momsp{N}([0,1])$, we denote by $\Pset_N(\CMom_N)$ the set of all measure $\mu \in \Pset([0,1])$, which are the solution of the following finite Hausdorff moment problem:
\begin{equation}
\cmom_{k/2}=\int_0^1x^{k/2}d\mu(x),\quad k= 0,1,\dots,N.
\end{equation}
If $\CMom_N=(\cmom_0,\hdots,\cmom_{N/2})^t$ belongs to the interior of $\Momsp{N}([0,1])$, we can show that the set $\Pset_N(\CMom_N)$ is infinite. Indeed, these results were shown in \cite{dette97} for integer moments, and its generalization for fractional is straightforward through the identification \eqref{identif:integer-frac-mom}.  Furthermore, the set of $\cmom_{(N+1)/2}(\mu)$, where $\mu \in \Pset(\CMom_N)$, is infinite and the maximum and the minimum of this set are different:

\begin{equation}
\cmom^-_{(N+1)/2}(\CMom_{N})=\underset{\mu\in\Pset(\CMom_N)}{\min}\left\{\cmom_{(N+1)/2}(\mu)\right\},\quad\cmom^+_{(N+1)/2}(\CMom_{N})=\underset{\mu\in\Pset(\CMom_N}{\max}\left\{\cmom_{(N+1)/2}(\mu)\right\}.
\end{equation}
 We define the canonical fractional moments as follows:
\begin{equation}
p_k=\dfrac{\cmom_{k/2}-\cmom^-_{k/2}(\CMom_{k-1})}{\cmom^+_{k/2}(\CMom_{k-1})-\cmom^-_{k/2}(\CMom_{k-1})}.
\end{equation}
The canonical moment vector belongs to $[0,1]^N$, which has a simple geometry compared to the moment space. By using the results obtained on canonical integer moments \cite{dette97}, we can write the algebraic relation between fractional moments and their corresponding canonical moments, in the case of $N=3$, by using the identification \eqref{identif:integer-frac-mom}:
\begin{equation}
  p_1=\dfrac{\mom_{1/2}}{\mom_0},\hspace{0.2cm}
  p_2=\dfrac{\mom_0\mom_1-\mom_{1/2}^2}{(\mom_0-\mom_{1/2})\mom_{1/2})},\hspace{0.2cm}
  p_3=\dfrac{(\mom_0-\mom_{1/2})(\mom_{1/2}\mom_{3/2}-\mom_1^2)}{(\mom_0\mom_1-\mom_{1/2}^2)
(\mom_{1/2}-\mom_1)}.
\end{equation}

\subsection{Principal representations - Gauss quadrature}

For a moment vector $\CMom_N$ in the interior of $\Momsp{N}$, the vector $\CMom^{\pm}_{N+1}=(\mom_0,\hdots,\mom_{N/2},\cmom^{\pm}_{(N+1)/2}(\CMom_N))$ belongs to the boundary of the moment space $\Momsp{N+1}$, and the measure set $\Pset_{N+1}(\CMom^{\pm}_{N+1})=\{\mu_{\pm}\}$ has only one element. 
The measure $\mu_+$ (resp. $\mu_-$) is called the upper (resp. lower) principal representation. In the case of integer moment vector $\widetilde{\tens{\cmom}}_N$, it was shown \cite{dette97} that the lower and upper principal representations ($\widetilde{\mu}_+$ and $\widetilde{\mu}_-$) can be expressed as sum of $n_s\leq (N+1)/2$ weighted delta-Dirac functions (we count the abscissas in the interior $]0,1[$ by $1$ and the ones in the extremity, $0$ or $1$, by $1/2$):
\begin{equation}
\widetilde{\mu}_{\pm}(r)=\sum\limits_{i=1}^{n_s}\widetilde{w}^{\pm}_i\delta_{r^{\pm}_i}(r).
\end{equation}
We recall that the subscript $\widetilde{\bullet}$ is used for integer moments and their corresponding measure, which are related to the fractional moments according to the identification \eqref{identif:integer-frac-mom}. 

In the case where $N$ is odd (corresponding to an even number of moments), the lower principal representation corresponds to the classical Gauss quadrature.
Let us remove the superscripts $-$ to simplify the notations.
Then, the weights $\widetilde{w}_i$ and the abscissas $r_i$ (which are in ]0,1[) can be computed from the moments $\widetilde{\tens{\cmom}}_N$, for example with the Product-Difference algorithm \cite{gordon1968}.
In other words, this algorithm (efficiently) solves the following non-linear system:
\begin{equation}
\widetilde{\cmom}_k=\sum\limits_{i=1}^{n_s}\widetilde{w}_ir_i^k,\hspace{0.3cm} k=0,1,\hdots,N.
\end{equation}
The identification \eqref{identif:integer-frac-mom} allows us to relate the quadrature in terms of fractional moments, to the quadrature in terms of integer moments: 
\begin{equation}
\cmom_{k/2} = \sum\limits_{i=1}^{n_s}{w}_ix_i^{k/2},\hspace{0.3cm} k=0,1,\hdots,N,
\label{eq:low_prin_frac}
\end{equation} 
where $x_i=(r_i)^2$ and $w_i=\tilde{w}_i$.
Thus, the Product-Difference algorithm allows to determine the weights and abscissas of the quadrature for fractional moments.

\section{Reconstruction through Entropy Maximization}
\label{SM:section_ME}
The Maximum Entropy (ME) problem reads as follows:
\begin{equation}
 \max\left\{ H[n]\!=\!-\!\int_0^1\!\ns(\size)\ln(\ns(\size))d\size\right\},\quad
 \mom_{k/2}\!=\!\int_0^1\!\size^{k/2}\ns(\size)d\size,\quad k\!=\!0,\hdots, N,
\label{SMeq:ME_optimization_problem}
\end{equation}
The existence and uniqueness of the ME solution is first shown, before giving an algorithm to compute the parameters of such a reconstruction.

\subsection{Existence and uniqueness of the maximum entropy solution}
\label{SM:ME-existence and uniqueness}

The following Theorem is shown here.
\begin{theorem}
If the vector $\Mom_{N}=(\mom_0,\mom_{1/2},\dots,\mom_{(N)/2})$ belongs to the interior of the Nth fractional moment space, then the constrained optimization \eqref{SMeq:ME_optimization_problem} problem admits a unique solution, which is in the following form:
\begin{equation}
\nME(\size)=\exp\left(-\lambda_0-\sum\limits_{i=1}^{N}\lambda_i\size^{i/2}\right).
\end{equation}
\end{theorem}
We mention that the case of the integer moments has been already treated in \cite{mead84}. 
Here, some ideas of this work are used, but the present proof is completely different and simplified. 
Indeed, Mead and Papanicolaou \cite{mead84} have used the monotonic properties of the moments, which is a characterization of the integer moment space, to prove existence of the ME solution. 
In our case, only the definition of the fractional moment space is used.
The proof is then done through two Lemmas.

\subsubsection{Uniqueness}
The first Lemma allows to show the uniqueness and to give the form of the ME reconstruction.
\begin{lemma}
If the constrained optimization problem \eqref{SMeq:ME_optimization_problem} admits a solution, then this solution is unique and can be written in the following form:
\begin{equation}
\nME(\size)=\exp\left(-\lambda_0-\sum\limits_{i=1}^{N}\lambda_i\size^{i/2}\right),
\end{equation}
where $\lambdaVec=(\lambda_0,..,\lambda_{N})^t\in\mathbb{R}^N$.
Moreover, this problem is equivalent to find an optimum of the potential function:
\begin{equation}
 G(\lambda_0,\hdots,\lambda_{N})=\int_{0}^{1}\exp(-\lambda_0-\sum_{i=1}^{N}\lambda_iS^{i/2})dS+\sum_{k=0}^{N-1}\lambda_{k}
\mom_{k/2}.
\label{SM-eq:Potential_function}
\end{equation}
\end{lemma}

\begin{proof}
The Lagrangian function associated to this standard constrained optimization problem is:\\
\begin{equation}
 L(\ns,\lambdaVec)=H[n]-(\lambda_0-1) 
\left(\int_0^1\ns(s)ds-\mom_{0}\right)-\sum\limits_{i=1}^{N}\left[\lambda_i 
\left(\int_0^1s^{i/2}\ns(s)ds-\mom_{i/2}\right)\right],
\end{equation}
where $\lambdaVec=(\lambda_0,..,\lambda_{N})$ is the vector of the Lagrange's multipliers.\\
Let us suppose that, for a given moment vector $\Mom_N$, there exists a density function $\nME$ which is the solution of the ME problem \eqref{SMeq:ME_optimization_problem}. So, there exists a 
vector $\lambEM$ for which the differential of the Lagrange function $ L(n,\lambdaVec)$ 
at the point $(\nME,\lambEM)$ vanishes:
\begin{equation}
\left\{
\begin{array}{rcl}
 DL(\nME,\lambdaVec)\cdot(h,\tens{0})&=&\int_0^1h(s)\left[-\ln(\nME(s))-\sum\limits_{i=0}^{N}\lambda_is^{i/2}\right]ds=0,\\
\frac{\partial L}{\partial \lambda_i}(\nME,\lambdaVec)&=&\int_0^1s^{i/2}\nME(s)ds- 
\mom_{i/2}=0,
\end{array}\right. ,
\label{eq:Existence-Lang}
\end{equation}
where $h$ is a positive distribution. Since the system \eqref{eq:Existence-Lang} is valid for all $h$, it yields:
\begin{equation}
\left\{
\begin{array}{rcl}
 \nME(\size)&=&\exp(-\lambda_0-\sum\limits_{i=1}^{N}\lambda_i\size^{i/2}),\\
\mom_{k/2}&=&\int_0^1s^{k/2}\exp(-\lambda_0-\sum\limits_{i=1}^{N}\lambda_is^{i/2})ds.
\end{array}\right.
\label{SysME}
\end{equation}
The problem then consists in finding a vector $\lambdaVec=(\lambda_0,..,\lambda_{N})$ in $\mathbb{R}^{N}$ which satisfies the moment equations in the system (\ref{SysME}).
This problem is equivalent to find an optimum of the potential function $G(\lambda_0,..,\lambda_{N})$ defined in \eqref{SM-eq:Potential_function}.

The Hessian matrix $H$ defined by $H_{i,j}=\frac{\partial^2 G}{\partial \lambda_i \partial \lambda_j}$ is a positive definite matrix, which ensures uniqueness of an eventual existing solution.
\end{proof}

\subsubsection{Existence}
The existence of such reconstruction is shown through the following Lemma.
\begin{lemma}
If the vector $\Mom_{N}=(\mom_0,\mom_{1/2},\dots,\mom_{(N)/2})$ belongs to the interior of the Nth fractional moment space, then the function $G$, defined in \eqref{SM-eq:Potential_function}, is a continuous function in $\mathbb{R}^N$, and goes to infinity when $||\lambdaVec||\rightarrow +\infty$.
\end{lemma}

\begin{proof}
Let us suppose that the last assertion is wrong, so there exists a sequence $(\lambdaVec^{(n)})_{n=0,1,..}$ such that $||\lambdaVec^{(n)}||\rightarrow +\infty$  when $n\rightarrow +\infty$ and $\sup\limits_n\left\{G(\lambdaVec^{(n)})\right\}<+\infty$ .\\
Hence, there exists   $A\in \mathbb{R}$ such that:
\begin{equation}
G(\lambdaVec^{(n)})=\int_{0}^{1}\exp(-\sum\limits_{i=0}^{N}\lambda^{(n)}_i\size^{i/2})d\size+
\sum\limits_{k=0}^{N}\lambda^{(n)}_{k}\mom_{k/2}<A.
\end{equation}
We write for each $n\in\mathbb{N}$, $\lambdaVec^{(n)}=\lambda^n(\alpha^{(n)}_0,\alpha^{(n)}_{0},\hdots,\alpha^{(n)}_{N})$,
such that $\sum_{i=0}^{N}(\alpha_i^{(n)})^2=1$ and $\lambda^{(n)}\rightarrow+\infty$.\\
Since the sequence $(\tens{\alpha}^{(n)})_{n=0,1..}$ is a bounded sequence, we can extract a convergent subsequence $(\tens{\alpha}^{\phi(n)})_n$, where $\phi:\mathbb{N}\rightarrow\mathbb{N}$ is an increasing function and:
 \begin{equation}
 \lim\limits_{n\rightarrow+\infty}\alpha^{\phi(n)}_i=\alpha_i.
 \end{equation} 
 To simplify the notations, we can directly consider   
that $\alpha^{(n)}_i\rightarrow\alpha_i$ when $n\rightarrow +\infty$.\\
We note by $Q^{(n)}(x)=\sum^{N}_{i=0}(\alpha_i^{(n)}x^{i/2})$ and $Q(x)=\sum_{i=0}^{N}\alpha_ix^{i/2}$.\\
\\
Since the vector $\Mom_N=(\mom_0,\mom_{1/2},..,\mom_{(N)/2})$ is a moment
vector, there exists a non-negative distribution function $f$ such that
$m_{k/2}=\int_0^1s^{k/2}f(s)ds$ for $k=0,\hdots N$, and
\begin{equation}
G(\tens{\lambda}^{(n)})=\int_0^1
\exp(-\lambda^{(n)}Q^{(n)}(s))ds+\int_0^1\lambda^{(n)}Q^{(n)}(s)f(s)ds\leq A.
\label{inequality}
\end{equation}
Since the first integral is positive
  \begin{equation}
  \int_0^1Q^{(n)}(s)f(s)ds\leq \dfrac{A}{\lambda^{(n)}}.
  \end{equation}
When $n$ tends to infinity, we get:
\begin{equation}
\int_0^1Q(s)f(s)ds\leq 0.
\label{ineq-proof1}
\end{equation}
We have $Q \neq 0$, $f\geq0$ and $f\neq0$, and since $Q$ is a continuous function, it follows from the inequality \eqref{ineq-proof1} that there exists $[a,b]\subset [0,1]$ in which $Q(s)\leq -B$ and $B>0$.
Since $Q^{(n)}$ converges uniformly to $Q$ in $[0,1]$, 
then, for all $s \in [a,b]$ and for the large enough values of $n$: 
\begin{equation}
Q^{(n)}(s)<-B/2.
\end{equation}
Using these results in the inequality (\ref{inequality}), as well as $\alpha^{n}_{k}\mom_{k/2}\geq -\mom_0$, we get:
\begin{equation}
\begin{array}{rcl}
A 
&\geq& \int_a^b
(exp(-\lambda^{(n)}Q^{(n)}(s)))ds+ \sum_{i=0}^{N}\lambda^{(n)}_{i}\mom_{i/2},\\[8pt]
&\geq&(b-a)exp(\lambda^{(n)}(\frac{B}{2}))-\lambda^{(n)}N\mom_0,\\
\end{array}
\end{equation}
In the limit when $n$ goes to infinity, we get the contradiction $+\infty\leq A$, thus concluding the proof.
\end{proof}


\subsection{Algorithm of the NDF reconstruction through the Entropy Maximization}
\label{SM:ME-algo}
The reconstruction of the NDF through the maximization of Shannon entropy goes back to finding the Lagrange's multipliers $\lambda_0\hdots\lambda_{N}$ such that:
\begin{equation}
\mom_{k/2}=\int_0^1\size^{k/2}\exp(-(\lambda_0+\sum\limits_{i=1}^{N}\lambda_i\size^{i/2}))dS, \quad k=0,\hdots ,N.
\end{equation}
Solving this nonlinear system is equivalent to optimize the convex function $G(\tens{\lambda})$. 
We solve the problem by using the Newton iterations as proposed in \cite{mead84}. 
The ME reconstruction is then done thanks to the Algorithm~\ref{EM-Algo}.
\begin{algorithm}[h!!]
\caption{ME algorithm}
\begin{algorithmic}
\STATE Choose initial guess of the vector $\lambdaVec$.
  \STATE $\delta_{k/2}\gets \mom_{k/2}-\int_0^1S^{k/2}\exp(-\sum\limits_{i=0}^{N}\lambda_iS^{i/2})dS$
\WHILE{$||\tens{\delta}||>\epsilon \mom_0$}
 \FOR{$i,j< N$}
  \STATE $H_{i,j}\gets \int_0^1S^{(i+j)/2}\exp(-\sum\limits_{i=0}^{N}\lambda_iS^{i/2})dS$
 \ENDFOR
 \STATE $\lambdaVec\gets \lambdaVec-\tens{H}^{-1}\cdot\tens{\delta}$
\FOR{$k<N$}
  \STATE $\delta_{k/2}\gets \mom_{k/2}-\int_0^1S^{k/2}\exp(-\sum\limits_{i=0}^N\lambda_iS^{i/2})dS$
   \ENDFOR
\ENDWHILE
\end{algorithmic}
\label{EM-Algo}
\end{algorithm}

The integral computations are done by using Gauss-Legendre quadrature. 
In \cite{mead84}, it is shown that $24$-point quadrature 
allows to accurately compute 
the different integral expressions involved in the Algorithm \ref{EM-Algo}.

\section{Transport scheme}
\label{app:Transport-scheme}
Here, first and second order schemes are given, which are directly extended from the schemes developed by Kah et al \cite{kah12} for the transport of integer moments.
For that, the canonical moments defined for fractional moments in section~\ref{appendix:fractional_moment_space} are used.

We choose to present the scheme in a two dimensional space to lighten the notations and we note $\U=(u,v)$. For the transport resolution in physical space, we use a dimensional splitting algorithm. In this context we consider a free transport in one direction (we present the $x$-direction here) 
of the droplets without the evaporation nor the drag force. 
\begin{equation}
 \begin{array}{r@{}lr}
\partial_t\mom_{k/2}+\partial_{x}(\mom_{k/2}u)&=0, &k=0,\hdots,3,\\
\partial_t(\mom_{1}\U)+\partial_{x}(\mom_{1}u\,\U)&=0, &\\
\end{array}
\label{eq:free_transport}
\end{equation}

The mathematical structure of the pressureless gas system leads to some singularities (known as $\delta$-shocks). These singularities occur when the monokinetic assumption is violated. This can happen when trajectory crossings take place and lead to particles accumulation in a very small volume. Following the idea of de Chaisemartin \cite{deChaisemartin2009}, Kah et al \cite{kah12}  developed a finite volume kinetic scheme for the EMSM model. We use the same approach to solve numerically the system \eqref{eq:free_transport}. In the following, we present briefly the main steps to derive the kinetic scheme for the system \eqref{eq:free_transport}:
\begin{enumerate}
\item We write the equivalent kinetic system to the pressureless system \eqref{eq:free_transport}, as it was proposed in \cite{bouchut2003}:
\begin{equation}
\left\{
 \begin{array}{c}
 \partial_tf+\partial_x(c_xf)=0\,\, \text{, and}\\
 f(t,x,\C,\size)=n(t,x,\size)\delta(\C-\U),
 \end{array}\right.
 \label{free_kinetic_transport}
\end{equation}
\item We use the finite volume discretization of the system \eqref{eq:free_transport}:
\begin{equation}
\begin{array}{rcl}
\Mom_i^{n+1}&=& \Mom_i^{n}-\dfrac{\deltat}{\Deltax}(F_{i+1/2}-F_{i-1/2}),\\[8pt]
\tens{p}_i^{n+1}&=&\tens{p}_i^{n}-\dfrac{\deltat}{\Deltax}(G_{i+1/2}-G_{i-1/2}),
\end{array}
\end{equation}
where the fluxes are expressed as function of the NDF: 
\begin{equation*}
\begin{pmatrix}
F^{\pm}_{i+1/2}\\
G^{\pm}_{i+1/2}
\end{pmatrix}
=\dfrac{1}{\deltat}\int_{t_n}^{t_{n+1}}\int_{0}^1\int_{\mathbb{R^2}}\begin{pmatrix}
1\\
\size^{1/2}\\
\size\\
\size^{3/2}\\
\size u\\
\size v
\end{pmatrix}f(t,x_{i+1/2},\C,\size)d\C d\size dt,
\end{equation*}
\item We split the fluxes in two integral parts: $F_{i+1/2} = F^+_{i+1/2}+F^-_{i+1/2}$ and $G_{i+1/2}=G_{i+1/2}^++G_{i+1/2}^-$.
The first (resp. the second) corresponds to the droplet of positive (resp. negative) velocity in $x$-direction. 
Then, we use the exact solution of the kinetic system \eqref{free_kinetic_transport}, to express the fluxes as function of the NDF at $t=t_n$.
\item Finally, the fluxes are expressed as function of a spatial reconstruction of the moments and velocities at $t=t_n$:
\begin{equation}
\begin{pmatrix}
F^{\pm}_{i+1/2}\\
G^{\pm}_{i+1/2}
\end{pmatrix}
=\dfrac{1}{\deltat}\int_{x_{i-1/2}}^{x_{i+1/2}}\begin{pmatrix}
\mom_0(t_n,x)\\
\mom_{1/2}(t_n,x)\\
\mom_1(t_n,x)\\
\mom_{3/2}(t_n,x)\\
\mom_1u(t_n,x)\\
\mom_1v(t_n,x)
\end{pmatrix}\mathds{1}_{\Sigma^{\pm}}(x)dx,
\end{equation}
where $\Sigma^{\pm}=\left\{x^{'},\pm(x_{i+1/2}-\deltat u(t_n,x^{'}))<\pm x^{'}\right\}$ 
\end{enumerate}
Here, we present two numerical kinetic schemes of first and second order, which follow this strategy.

\subsection{First order scheme}
For a first order scheme, we consider a constant piecewise reconstruction for the moments and the velocity. Then the fluxes become:
\begin{equation}
\begin{pmatrix}F_{i+1/2}\\
G_{i+1/2}\end{pmatrix}=\begin{pmatrix}
m^n_{0,i}\\
m^n_{1/2,i}\\
m^n_{1,i}\\
m^n_{3/2,i}\\
m^n_{1,i}u^n_{0,i}\\
m^n_{1,i}v^n_{i}\\
\end{pmatrix}max(u_{i}^n,0)+\begin{pmatrix}
m^n_{0,i+1}\\
m^n_{1/2,i+1}\\
m^n_{1,i+1}\\
m^n_{3/2,i+1}\\
m^n_{1,i+1}u^n_{0,i+1}\\
m^n_{1,i}v^n_{i+1}\\
\end{pmatrix}min(u_{i+1}^n,0),
\label{moment_flux_o1}
\end{equation}


\subsection{Second order scheme}
Kah et al \cite{kah12} developed a realizable second order kinetic scheme. They showed that the canonical moments (in the context of integer moments) are transported variables. Therefore, these quantities satisfy a maximum principle. Since, the canonical moments live in the simple space $[0,1]^N$ ($N=3$ in our case). The authors proposed to use linear reconstruction of the canonical moments to design a high order scheme, instead of reconstructing directly the moments, which belong to a complex space. We adopt the same approach with some adaptations for the fractional moments. After reconstructing the variables (velocity and canonical moments) the fluxes are computed by a simple integration.
\subsubsection{Reconstruction}
 The reconstructed variables are the moment $\mom_0$, the canonical moments (defined in part \ref{appendix:fractional_moment_space}) $p_1, p_2, p_3$ and the velocity.
 \begin{equation}
 \left\{
 \begin{array}{rcl}
 \mom_0(x)&=&\mom_{0,i}+ Dm_{0,i}(x-x_i),\\
 p_1(x)&=&\overline{p_{1,i}}+ Dp_{1,i}(x-x_i),\\
 p_2(x)&=&\overline{p_{2,i}}+ Dp_{2,i}(x-x_i),\\
 p_3(x)&=&\overline{p_{3,i}}+ Dp_{3,i}(x-x_i),\\
 u(x) &=& \overline{u_i}+Du_{i}(x-x_i),\\
 v(x) &=& \overline{v_i}+Dv_{i}(x-x_i),\\
 \end{array}\right.
 \end{equation}
 where $x\in [x_{i-1/2},x_{i+1/2}]$. Generally the quantities with the bar are different from the cell averaged quantities $p_{k,i}$, $u_{k,i}$ and $v_{k,i}$ and they are determined depending on the slopes and the following conservation properties:
 
 \begin{equation}
  \begin{array}{rcl}
  \mom_{1/2,i}^n&=&\dfrac{1}{\Deltax}\int_{x_{i-1/2}}^{x_{i+1/2}}\mom_0(x)p_1(x)dx,\\[8pt]
  \mom_{1,i}^n &=& \dfrac{1}{\Deltax} \int_{x_{i-1/2}}^{x_{i+1/2}}\mom_0(x)p_1(x)[(1-p_1)p_2+p_1](x)dx,\\[8pt]
  \mom_{3/2,i}^n&=&\dfrac{1}{\Deltax} \int_{x_{i-1/2}}^{x_{i+1/2}}\mom_0p_1\left\{(1-p_1)(1-p_2)p_2p_3+[(1-p_1)p_2+p_1]^2\right\}(x)dx,\\[8pt]
  m^n_{1,i}u_i^n&=&\dfrac{1}{\Deltax}\int_{x_{i-1/2}}^{x_{i+1/2}}\mom_0(x)p_1(x)[(1-p_1)p_2+p_1](x)u(x)dx.
  \end{array}
  \label{eq:conserv-property}
 \end{equation}
 Compared to the expressions developed in the case of the EMSM model, only the last integral expression is different. In fact, the velocity is weighted with the moment $\mom_1$ for both models, but in the new model, $\mom_1$ acts as a second order moment. For this reason, the expression of the moment $\mom_1$ as function of the canonical moment is different from the one in the case of integer moments.\\
 
 Kah el al. \cite{kah12} show that the bar terms can be written as follows:
 \begin{equation}
 \begin{array}{rcl}
 \overline{p_{k,i}}&=&a_{k,i}+b_{k,i}Dp_{k,i},\\
  \overline{u_{i}}&=&a_{u,i}+b_{u,i}Du_i,\\
 \end{array}
 \label{eq:bar-simplif}
 \end{equation}
where for each $k$, $a_{k,i}$ and $b_{k,i}$ are independent of $Dp_{k,i}$  , and $a_{u,i}$ and $b_{u,i}$ are independent of $Du_i$.\\

 \subsubsection{Slope limitation}
 In order to satisfy the maximum principle for the transported quantities (the canonical moment and the velocity) and the positivity of the number density $\mom_0$, the slopes should be calculated carefully. Following the development done in \cite{kah12}, the slopes are calculated as follows:
 \begin{equation}
\begin{array}{ll}
 Dm_{0,i}&=\phi(\mom_{0,i-1}^n,\mom_{0,i}^n,\mom_{0,i}^n)\min\left(\dfrac{|\mom_{0,i+1}^n-\mom_{0,i}^n|}{\Deltax},\dfrac{|\mom_{0,i}^n-\mom_{0,i-1}^n|}{\Deltax},\dfrac{2\mom_{0,i}^n}{\Deltax}\right),\\
Dp_{k,i}&=\phi(p_{k,i-1}^n,p_{k,i}^n,p_{k,i+1}^n)\min\left(\dfrac{|p_{k,i+1}^n-a_{k,i}|}{\Deltax+2b_{k,i}},\dfrac{|a_{k,i}-p_{k,i-1}^n|}{\Deltax-2b_{k,i}}\right),\\
Du_{i}&=\phi(u_{i-1}^n,u_{i}^n,u_{i+1}^n)\min\left(\dfrac{|u_{i+1}^n-u_{i}^n|}{\Deltax+2b_{u,i}},\dfrac{|u_{i}^n-u_{i-1}^n|}{\Deltax-2b_{u,i}},\dfrac{1}{\deltat}\right),
\end{array}
 \label{eq:solpes-relations}
 \end{equation}
 where $\phi(a,b,c)=1/2(sgn(b-a)+sgn(c-b))$.\\
 
 Using the equations \eqref{eq:conserv-property},\eqref{eq:bar-simplif} and \eqref{eq:solpes-relations} the slopes and the bar variables can be expressed as functions of the current and neighbor cell variables. However, these algebra relations are quite heavy. Therefore, their calculation is achieved using Maple software.\\ 
 \subsubsection{Fluxes Computation}
After computing the slopes and the bar variables, the fluxes can be computed as follows:
 \begin{equation}
\begin{pmatrix}
F^+_{i+1/2}\\
G^+_{i+1/2}
\end{pmatrix}
=\dfrac{1}{\deltat}\int_{x^L_{i+1/2}}^{x_{i+1/2}}\mom_0\begin{pmatrix}
1\\
p_1\\
p_1[(1-p_1)p_2+p_1]\\
p_1\left\{(1-p_1)(1-p_2)p_2p_3+[(1-p_1)p_2+p_1]^2\right\}\\
p_1[(1-p_1)p_2+p_1]u\\
p_1[(1-p_1)p_2+p_1]v
\end{pmatrix}dx,
\end{equation}
and
\begin{equation}
\begin{pmatrix}
F^-_{i+1/2}\\
G^-_{i+1/2}
\end{pmatrix}
=-\dfrac{1}{\deltat}\int_{x_{i+1/2}}^{x^R_{i+1/2}}\mom_0\begin{pmatrix}
1\\
p_1\\
p_1[(1-p_1)p_2+p_1]\\
p_1\left\{(1-p_1)(1-p_2)p_2p_3+[(1-p_1)p_2+p_1]^2\right\}\\
p_1[(1-p_1)p_2+p_1]u\\
p_1[(1-p_1)p_2+p_1]v
\end{pmatrix}dx,
\end{equation}

such that
\begin{equation}
\begin{array}{rcl}
x_{i+1/2}^L&=&x_{i+1/2}-\deltat\dfrac{(\bar{u}_i+\frac{\Deltax}{2}Du_i)_+}{1+\deltat Du_i},\\[8pt]
x_{i+1/2}^R&=&x_{i+1/2}-\deltat\dfrac{(\bar{u}_{i+1}-\frac{\Deltax}{2}Du_{i+1})_+}{1+\deltat Du_{i+1}}.\\
\end{array}
\end{equation}
The expressions inside the integrals are polynomial functions of $x$ of order up to $6$, their calculation can be achieved by using four points of the Gauss-Legendre quadrature.

\bibliographystyle{siamplain}
\bibliography{biblio}
\end{document}